\documentclass{amsart}

\usepackage{amsmath, amsthm, amssymb, amsfonts,enumerate}
\usepackage[foot]{amsaddr}
\usepackage[normalem]{ulem}
\usepackage[pdfusetitle, colorlinks=true]{hyperref}
\usepackage{color}
\usepackage{comment}
\theoremstyle{plain}

\newtheorem{theorem}{Theorem}[section]

\newtheorem{lemma}[theorem]{Lemma}
\newtheorem{lem}[theorem]{Lemma}

\newtheorem{corollary}[theorem]{Corollary}

\newtheorem{conjecture}[theorem]{Conjecture}

\newtheorem{proposition}[theorem]{Proposition}

\theoremstyle{definition}

\newtheorem{definition}[theorem]{Definition}

\theoremstyle{remark}
\newtheorem{remark}[theorem]{Remark}

\DeclareSymbolFont{AMSb}{U}{msb}{m}{n}
\DeclareMathSymbol{\N}{\mathalpha}{AMSb}{"4E}
\DeclareMathSymbol{\R}{\mathalpha}{AMSb}{"52}
\DeclareMathSymbol{\Z}{\mathalpha}{AMSb}{"5A}
\DeclareMathSymbol{\D}{\mathalpha}{AMSb}{"44}
\DeclareMathSymbol{\s}{\mathalpha}{AMSb}{"53}

\newcommand{\CD}{\mathrm{CD}}
\newcommand{\RCD}{\mathrm{RCD}}
\newcommand{\CAT}{\mathrm{CAT}}
\newcommand{\CBA}{\mathrm{CBA}}

\newcommand{\eps}{\varepsilon}
\renewcommand{\phi}{\varphi}

\newcommand{\e}{e}

\DeclareMathOperator{\argmin}{argmin}

\DeclareMathOperator{\Ric}{Ric}

\DeclareMathOperator{\Hess}{Hess}

\DeclareMathOperator{\vol}{vol}

\renewcommand{\tilde}{\widetilde}

\newcommand{\cvertex}{\mathbf{0}}

\DeclareMathOperator{\diam}{diam}

\newcommand{\lk}{\kappa}
\newcommand{\uk}{K}

\definecolor{grey}{rgb}{0.800781, 0.800781, 0.800781}

\def\S2k{\mathbb{S}^2_\kappa}

\def\co{\colon\thinspace}
\def\SS{\mathbb S}
\def\const{\mathrm{const}}
\DeclareMathOperator{\curv}{curv}
\def\D{\mathbf D}

\def\vk{\theta}
\def\Vk{\Theta}

\DeclareMathOperator{\injrad}{Injrad}
\DeclareMathOperator{\Isom}{Isom}

\DeclareMathOperator{\Aut}{Aut}

\DeclareMathOperator{\Scal}{Scal}
\newcommand{\bCAT}{\boldsymbol\partial}

\DeclareMathOperator{\Rm}{Rm}

\def\Mid{M_{i,\delta} }
\newcommand\Rk[1]{\mathcal R_{#1}}

\def\onto{\twoheadrightarrow}

\author{Vitali Kapovitch}\thanks{University of Toronto, vtk@math.utoronto.ca}

\title{Mixed curvature almost flat manifolds}
\begin{document}
\begin{abstract} 
We prove a mixed curvature analogue of Gromov's almost flat manifolds theorem for upper sectional and lower Bakry\textendash Emery Ricci curvature bounds.
\end{abstract}
\maketitle
\section{Introduction}
A celebrated Almost Flat Manifolds  theorem of Gromov \cite{Gr-almost-flat} says that for any natural $n$ there exists a constant $\eps(n)$ such that if a closed manifold $M^n$ admits a Riemannian metric of $|\sec|\le 1$ and $\diam\le \eps(n)$ then $M$ is finitely covered by a nilmanifold.  This was later refined by Ruh \cite{Ruh-almost-flat} to show that  $M$ must be infranil. 

Here by an infranilmanifold $M$ we mean the quotient of a simply-connected nilpotent Lie group $G$ by the action of a torsion free discrete subgroup $\Gamma$ of the semidirect product of $G$ with a compact subgroup of $\Aut(G)$. If $M$ is compact then a finite index subgroup $\hat \Gamma$ of $\Gamma$ lies entirely in $G$ and hence $M$ is finitely covered by the homogeneous space $G/\hat\Gamma$. Such homogeneous spaces are called nilmanifolds.

Any closed infranilmanifold admits a metric with  $|\sec|\le 1$ and $\diam\le \eps$ for any $\eps>0$. Gromov called such manifolds \emph{almost flat}.  
The purpose of this paper is to extend Gromov's theorem to the case where the lower sectional curvature bound    is replaced by the weaker assumption of a lower Bakry\textendash Emery Ricci curvature bound.

Let us elaborate on our motivation for considering this condition.

By a key observation of Gromov the class of $n$-dimensional manifolds with $\Ric\ge \lk$ is precompact in pointed Gromov\textendash Hausdorff topology for any fixed $\lk,n$. It is therefore natural to consider converging sequences of elements in this class and study their limits and their geometric and topological relationship to the elements of the sequence. The above class contains several natural subclasses. The most restrictive of these is manifolds with two sided sectional curvature bounds $\lk\le \sec\le K$. Gromov's Almost Flat Manifolds Theorem  deals with the special case when the limit is a point. A more general theory  dealing with arbitrary limits was later developed by Cheeger, Fukaya and Gromov \cite{Che-Gro-col-1,Che-Gro-col-2,CFG, Rong-survey}.

One can relax the above bounds by getting rid of the upper sectional curvature bound. The limits of manifolds in this class are Alexandrov spaces, i.e. complete inner metric spaces for which Toponogov triangle comparison holds. Manifolds that can converge to a point under a lower sectional curvature bound are called almost nonnegatively curved. These are more general than Gromov's almost flat manifolds (in particular any nonnegatively curved closed manifold is such). While their geometry and topology is not completely understood it was shown in ~\cite{KPT} that any almost nonnegatively curved $n$-manifold virtually fibers over a nilmanifold with simply connected fibers and  its fundamental groups contain a nilpotent subgroup of index $\le C(n)$.

In the most general setting when one only assumes a lower Ricci bound the limits are so-called $\RCD(\lk,n)$ metric measure spaces. These are metric measure spaces which are infinitesimally Hilbertian and satisfy the curvature dimension condition $CD(\lk,n)$ of Sturm\textendash Lott\textendash Villani. While they serve as motivation we only make very limited use of the theory of $\RCD$ spaces in the present paper.  The reader is referred to  \cite{ambrosio-gigli-savare-survey, villani-survey, Kap-Ket-18, Kap-Kell-Ket-19} for a  review and background on $\RCD$ spaces.

The class  of $\RCD(\lk,n)$ spaces contains a natural subclass which is also closed under pmGH convergence given by spaces which in addition are $\CAT(\uk)$ for some fixed $\uk\in\R$. Moreover, in analogy with Cheeger\textendash Fukaya\textendash Gromov theory it's natural to look at a more general class of  $\RCD(\lk,n)$  which are also $\CBA(\uk)$, i.e. which are $\CAT(\uk)$ only locally. 
Local structure of such spaces was investigated in \cite{Kap-Ket-18, Kap-Kell-Ket-19}. It turned out to bear many similarities with the local structure of metric spaces with two sided sectional curvature bounds in Alexandrov sense.

Just like the class of manifolds with two sided sectional curvature bounds the class of  $\RCD(\lk,N)+\CBA(\uk)$ spaces is precompact in pmGH topology.
Given the similarity of the structure theory between  $\RCD(\lk,N)+\CBA(\uk)$ spaces and its subclass of spaces with two sided Alexandrov curvature bound it is natural to wonder to what extent the convergence theory for the latter class generalizes to the former.

In the smooth settings a complete weighted Riemannian manifold   $(M^n, g, e^{-f}\cdot \mathcal H_n)$ (where $ \mathcal H_n$ is the n-dimensional Hausdorff measure) with a smooth function $f$  is  $\RCD(\lk,N)$  iff $n\le N$ and the Bakry\textendash Emery Ricci curvature $\Ric_{f,N}$ is bounded below by $\lk$ ~\cite{saintfleur,lottvillani,erbarkuwadasturm}. Recall that the 
 Bakry\textendash Emery Ricci tensor  $\Ric_{f,N}$ is given by $\Ric_{f,N}=\Ric+\Hess f-\frac{df\otimes df}{N-n}$. If $N=n$ then $f$ must be  constant and the Bakry\textendash Emery Ricci curvature reduces to the usual Ricci curvature.

The above discussions suggests that it is natural to study the following class of manifolds.
Given $N\ge 1, \lk,\uk\in \R$ consider the class $\mathfrak M_\lk^{\uk,N}$ of complete smooth weighted Riemannian manifolds $(M^n, g, e^{-f}\cdot \mathcal H_n)$  satisfying 
\begin{equation}\label{eq:cd+cba}
\text{$(M^n,g,e^{-f}\cdot \mathcal H_n)$  has $\Ric_{f,N}\ge \lk$ and  $\sec\le \uk$.}
\end{equation}

The class $\mathfrak M_\lk^{\uk,N}$  is precompact in pointed measured Gromov\textendash Hausdorff topology.  Given a sequence $M_i^n\in \mathfrak M_\lk^{\uk,N}$ pmGH converging to a space $(X,d,m)$ we wish to study the structure of this convergence, i.e. the structure of the limit spaces and the relationship between $M_i$ and $X$.

As in the situations described above a key basic case of such convergence is convergence to a point which is equivalent to having a sequence  $M_i^n\in \mathfrak M_\lk^{\uk,N}$ with $\diam M_i\to 0$. By rescaling this is equivalent to having a sequence $(M^n_i,g_i,e^{-f_i}\cdot  \mathcal H_n)$ and $0<\eps_i\to 0$ such that $\sec M_i\le \eps_i, \Ric_{f_i,N}(M_i)\ge -\eps_i$ and $\diam M_i\le \eps_i$.

The main goal of this paper is to show that  Gromov's Almost Flat Manifolds Theorem generalizes to this setting. 
Namely, we prove

\begin{theorem}[Main Theorem]\label{thm:almostflat}
For any $1<N<\infty$ there exists $\eps=\eps(N)>0$ such that the following holds.

If $(M^n,g,e^{-f}\cdot \mathcal H_n)$ is a weighted closed Riemannian manifold with $n\le N$ and $\sec\le\eps, \diam\le \eps $ and $\Ric_{f,N}\ge -\eps$ then $M^n$ is diffeomorphic to an infranilmanifold.
\end{theorem}

As was mentioned above  in the case $n=N, f=\const$ the bound $\Ric_{f,N}\ge \lk$ reduces to  $\Ric\ge \lk$ which combined with the upper sectional curvature bound  trivially implies that 
 $\sec   \ge \lk-(n-2)\uk$. In particular if $\sec\le \eps$ and $\Ric\ge -\eps$ then $\sec\ge -(n-1)\eps$. Therefore if $N=n$ in Theorem~\ref{thm:almostflat} then the result trivially follows from the original Almost Flat Manifolds Theorem.
However, in the general case $n<N$ it's easy to see that spaces satisfying \eqref{eq:cd+cba} have no uniform lower sectional curvature bounds. In Lemma~\ref{lemma-ex}  for any $N>2$ we construct a sequence $(T^2,g_i,e^{-f_i}\cdot \mathcal H_n)$  with $\sec_{g_i}\le 1/i$, 
$\Ric_{f_i,N}\ge -1/i$, $\diam(g_i)\le 1$ and minimum of sectional curvature going to $-\infty$ as $i\to\infty$.

 It is well known that  a solvable subgroup of the fundamental group of a closed manifold of nonpositive curvature must be virtually abelian \cite[Theorem 7.8]{BH99}. Therefore if the upper sectional curvature bound in Theorem 1 is 0 then the nilpotent Lie group which is the universal cover of $M$ is abelian i.e. it's isomorphic to $\R^n$. This yields
 
 \begin{corollary}\label{main:cor}
 For any $1<N<\infty$ there exists $\eps=\eps(N)>0$ such that the following holds.

If $(M^n,g,e^{-f}\cdot \mathcal H_n)$ is a weighted closed Riemannian manifold with $n\le N, \sec\le 0, \diam\le \eps $ and $\Ric_{f,N}\ge -\eps$ then $M^n$ admits a flat metric. In particular $M$ is diffeomorphic to the quotient of $\R^n$ by a Bieberbach group.

 \end{corollary}

Let us describe the strategy of the proof of the main theorem. The main part of the proof is to show that under the assumptions of the theorem $M$ is aspherical (i.e. its universal cover is contractible)  and the universal cover has large injectivity radius. This follows from 
 Theorem ~\ref{asphericity-thm} (Asphericity Theorem).
 The proof of the Asphericity Theorem occupies the bulk of the paper. Let us note that the same statement also takes a large part of Gromov's original proof of the Almost Flat Manifolds theorem. Our situation is harder as the assumptions are weaker and some tools used in Gromov's proof are not available.

A key technical result we prove first is 
a fibration theorem similar to the fibration theorem in ~\cite{CFG} under the assumptions of two-sided sectional curvature bounds. More generally, Yamaguchi  showed \cite{Yam} that the fibration theorem holds under just lower sectional curvature bound. Before we state the fibration theorem we need the following definition.
 
 \begin{definition}
A smooth map $f\co M\to N$ between smooth Riemannian manifolds is an $\eps$-almost Riemannian submersion if for any $p\in M$ the map $d_pf\co T_pM\to T_{f(p)}N$ is onto and satisfies that for any $v\in T_pM$ orthogonal to $\ker d_pf$ it holds that
\[
\frac{1}{1+\eps}|v|\le |d_pf(v)|\le (1+\eps)|v|
\]
\end{definition}

 Note that we \emph{do not} require $f$ to be onto. If $f$ is onto we will indicate that by a double headed arrow $f\co M\onto N$. For $p\in M$ and $v\in T_pM$ we call $v$ \emph{vertical} if $v\in \ker d_pF$, i.e. $v$ is tangent to the fiber of $f$ passing through $p$.
 We call $v$ \emph{horizontal} if $v$ is orthogonal to $\ker d_pF$.
 Any $\eps$-almost Riemannian submersion is $(1+\eps)$-Lipschitz. 
 Further if both $M,N$ are complete then any curve $\gamma$ in $N$ starting at $f(p)$ admits a horizontal lift $\bar\gamma$ starting at $p$ such that $L(\bar\gamma)\le (1+\eps)L(\gamma)$.

 In particular $f$ is $(1+\eps)$-co-Lipschitz, i.e. $f(B_r(p))\supset  B_{\frac{r}{1+\eps}}(f(p))$ for all  $r$.

  We will adopt the notation that $\R_{+}\times \R^k\ni (\varepsilon,x)\mapsto  \vk(\varepsilon|x)\in \R_{+}$ denotes a non-negative function satisfying that, for any fixed $x=(x_1,\ldots,x_k)$, $\lim_{\varepsilon\to 0} \vk(\varepsilon|x)=0 $. Similarly $\Vk(\eps|x)$ will denote a function  $\R_{+}\times \R^k\to \R_{+}$ such that  $x\in \R^k$, $\lim_{\varepsilon\to 0} \Vk(\varepsilon|x)=\infty$.

We prove
  
\begin{theorem}[Compact Fibration theorem]\label{fibration-thm}

Let $M_i^n\to B^m$ in  pointed Gromov\textendash Hausdorff topology where $B^m$ is a smooth closed Riemannian manifold  and all $(M_i, g_i, e^{-f_i}\cdot  \mathcal H_n)$ satisfy \eqref{eq:cd+cba}. 

Let $h_i\co M_i\to B$ be $\eps_i$-GH-approximations with $\eps_i\to 0$ as $i\to \infty$.

Then there exist $\delta_i\to 0$ as $i\to\infty$ and smooth maps $\pi_i\co M_i\to B$ such that 
$\pi_i\co M_i\to B$ is a $\delta_i$-almost Riemannian submersion  which is $\delta_i$-uniformly close to $h_i$ for all $i$.

\end{theorem}

Note that a proper submersion  is a fiber bundle, hence the name of the theorem.
As a consequence of the proof of the Main Theorem we also prove that the fibers in the Fibration Theorem are \emph{homeomorphic} to infranilmanifolds (Theorem~\ref{infranil-fibers}).


A key ingredient in the proof of the Main Theorem overall and of the Fibration Theorem in particular is the following observation of Gromov which is also central to the original proof of Gromov's Almost Flat Manifolds Theorem. Given a complete manifold $(M^n, g)$ of $\sec\le \eps$ let $p\in M$ and let $\exp_p\co T_pM\to M$ be the exponential map. The condition $\sec\le \eps>0$ implies that the conjugate radius of $M$ is at least $\frac{\pi}{\sqrt{\eps}}$ and hence $\exp_p$ is a local diffeomorphism on $B_{\frac{\pi}{\sqrt{\eps}}}(\cvertex)\subset T_pM$. We can therefore pull back the metric $g$ from $M$ to $B_{\frac{\pi}{\sqrt{\eps}}}(\cvertex)$. The resulting non-complete manifold $\hat M=(B_{\frac{\pi}{\sqrt{\eps}}}(\cvertex)\,\exp_p^*(g))$ still has $\sec\le \eps$ but also has injectivity radius at $\cvertex$ equal to $\frac{\pi}{\sqrt{\eps}}$. 

The projection map $\exp_p\co \hat M\to M$ is a so-called pseudo-cover with the pseudo-group equal to the set of short geodesic loops at $p$ acting as deck transformations. The advantage of working with the pseudo-cover $\hat M$ is that even if $M$ is collapsed the pseudo-cover $\hat M$ is not collapsed. This means that all collapsing happens entirely  due to the presence of short geodesic loops. This reduces the local structure of collapsed manifolds with upper sectional and lower Bakry\textendash Emery bounds to the study of local pseudo-covers and the corresponding pseudo-group actions.

With the Fibration Theorem at our disposal the proof of the Asphericity Theorem proceeds along the following lines.

Given a  sequence $(M^n_i,g_i, e^{-f_i}\cdot  \mathcal H_n)$ with $\sec_{M_i}\le \eps_i,\Ric^{M_i}_{f_i,N}\ge -\eps_i$ and $\diam(M_i)\le \eps_i$ we rescale the sequence to have diameter 1 and pass to the limit. Working with the pseudo-covers one can show that passing to actual covers of uniformly bounded order we can assume that the limit is a flat torus $T^{l_1}$ with $l_1>0$. This implies that $\pi_1(M_i)$ surjects onto $\pi_1(T^{l_1})=\Z^{l_1}$. Going back to the original unrescaled sequence, using the surjection $\pi_1(M_i)\to \Z^{l_1}$  we can now find some finite covers $\bar M_i$ of $M_i$ which converge to a space of diameter 1 which again must be a flat torus $T^{q_1}$. Using the Fibration Theorem we get an almost Riemannian submersion $\pi_i\co \bar M^i\to T^{q_1}$. Taking the covers corresponding to the kernel of $(\pi_i)_*$ on the fundamental group we get almost Riemannian submersions $\check \pi_{i,1}\co \check M_{i,1}\to \R^{q_i}$ with small fibers. Note that due to contractibility of $\R^{q_i}$ the inclusion of the fibers of $\check \pi_{i,1}$ is a homotopy equivalence. 
We can now repeat the same process but with the fibers of $\check \pi_{i,1}$ instead of $M^i$ to get some finite covers of $\check M_{i,1}$ that converge to $\R^{q_1}\times T^{q_2}$ with $q_2>0$. Again taking the covers corresponding to the kernels of the induced maps on $\pi_1$ we get covers $\check M_{i,2}$  that converge to $\R^{q_1+q_2}$ etc.

This process will terminate after finitely many steps because the dimensions of all the tori showing up in the construction are positive. At the final step we get 
almost Riemannian submersions $\check \pi_{i,k}\co \check M^n_{i,k}\to \R^n$ with small fibers such that still the inclusions of the fibers into the total spaces are homotopy equivalences. For dimension reasons the fibers must be 0-dimensional and since they are connected they must be points.

The above procedure is formalized in the Induction Theorem (Theorem~\ref{thm:induction}).

Once the Asphericity Theorem is proven we use a Ricci flow argument to show that the metric on a manifold satisfying the assumptions of the Main Theorem for sufficiently small $\eps$ can be deformed into an almost flat metric in the original sense of Gromov. This is done as follows.

We use the following construction due to Lott   \cite{lobaem,Lott-optimal}.  Look at the warped product metric on  $E=\bar M\times_{e^{-f/q}}T^q_\eps$ where $q=N-n$ (we can assume that $N$ is an integer by increasing by at most 1) and $T^q_\eps$ is a standard flat torus of diameter $\eps$.
Then  the horizontal Ricci tensor of $E$ is given by the Bakry\textendash Emery Ricci tensor on $M$. Moreover we show that our curvature assumptions on $M$ imply certain semiconcavity of the warping function which yields that $E$ has small diameter and almost nonnegative Ricci curvature. 
The Asphericity Theorem implies that large balls in the universal cover of $E$ are Lipschitz close to the corresponding balls in $\R^N$. In particular they have almost the same isoperimetric constants.
Now a Ricci flow argument using Perelman's Pseudo-locality shows that time 1 Ricci flow turns $E$ into an almost flat manifold in
 Gromov's sense. The Ricci flow preserves the warped product structure and hence we get an almost flat metric on $M$ as well. 
 
 The paper is structured as follows. In Section~\ref{sec:pseudo-covers} we give background on pseudo-groups and pseudo-covers. In Section \ref{sec:  RCD+CAT} we study the topology and geometry of almost flat mixed curvature manifolds with large injectivity radius. In Section ~\ref{sec:fibration thm} we prove the Fibration Theorem. In Section~\ref{sec: Induction thm} we prove the Induction Theorem and the Asphericity Theorem.
 In Section~\ref{sec:main-homeo} we prove the Main Theorem up to homeomorphism. In Section~\ref{sec:ricci flow} we prove the Main Theorem up to diffeomorphism. In Section~\ref{sec:example} we construct an example of a sequence of mixed curvature almost flat metrics $g_i$ on $T_2$ such that the lower sectional curvature bound of $g_i$ converges to $-\infty$ as $i\to\infty$.
 

The author would like to thank   Igor Belegradek, John Lott, Robert Haslhofer, Christian Ketterer  and Anton Petrunin for helpful conversations and suggestions. The author is grateful to Richard Bamler for suggesting the method of getting the upper diameter bound in the proof of Theorem ~\ref{thm:almost-flat-Ricci-flow}. The author is funded by a Discovery grant from NSERC.

\tableofcontents
\section{Pseudo-covers}\label{sec:pseudo-covers}

The following definition is due to Gromov \cite{Gr-almost-flat}.\footnote{Some texts use a different notion of a pseudo-group.}
\begin{definition}\label{defn:pseudo-group}

A pseudo-group is a set $G$ with a binary operation $\star\co G\times G\supset A\to G$ which satisfies the following properties
\begin{enumerate}
\item There is a unique $e\in G$ such that for any $g\in G$ the products $e\star g$ and $g\star e$ are defined and are equal to $g$.
\item for any $g_1,g_2,g_3\in G$ it holds that $(g_1\star g_2)\star g_3=g_1\star (g_2\star g_3)$ provided both products are defined.
\item for any $g\in G$ there is $h\in G$ such that $h\star g=g\star h=\e$.
\end{enumerate}
\end{definition}

Any group is obviously a pseudo-group and more generally any subset $G$ in a group $H$ which contains the identity and is closed under taking inverses is a pseudo-group with respect to the group operation on $H$.
Homomorphisms of pseudo-groups are defined in an obvious way.

Any pseudo-group $(G,\star)$ canonically generates a group $\bar G$ as follows. The group $\bar G$ is defined to be the quotient of the free group $\langle a_g | g\in G\rangle$ on elements of $G$  modulo  the normal subgroup $N$ generated by all relations of the form $a_{g_1}\cdot a_{g_2}\cdot a_{g_3}^{-1}$ whenever $g_1,g_2,g_3\in G$ satisfy $g_1\star g_2=g_3$.
There is an obvious canonical homomorphism  $\phi\co G\to \bar G$ given by  $g\mapsto a_g $ mod $N$.
This homomorphism obviously satisfies the following universal property: Given a homomorphism $\rho\co G\to H$ where $H$ is a group there is a unique group homomorphism $\bar \rho \co \bar G\to H$ such that $\rho=\bar\rho \circ \phi$. The map  $\phi$ need not be 1-1 (see Remark~\ref{rem: non-injective} below).
 However, given an element $g\in G$ if we can find a homomorphism  $\rho\co G\to H$ where $H$ is a group such that $\rho(g)\ne e$ then the above universal property implies that $\phi(g)\ne e $ as well. We will make use of this simple but important fact in the proof of the main theorem.

Next we outline the construction of pseudo-groups and pseudo-covers that  naturally arise in the context of manifolds with two sided curvature bounds. They play a key role both in  the original proof of Gromov's Almost Flat manifolds theorem and in the proof of the Theorem~\ref{thm:almostflat}.
This construction is explained in detail in \cite[Section 2]{BuKa}. 

Let $(M^n, g)$ be complete with $\sec\le \uk$. Let $p\in M$ and let $\exp\co T_pM\to M$ be the exponential map at $p$. Since $\sec\le \uk$ the conjugate radius of $M$ is infinite if $\uk\le 0$ and is at least $\frac{\pi}{\sqrt{\uk}}$ if $\uk >0$.  
Set

\[
\Rk{\uk}:=\begin{cases}
\infty\quad\text{ if } \uk\le 0\\
\frac{\pi}{2\sqrt{\uk}}\quad\text{ if } \uk> 0
\end{cases}
\]

Then $\exp$ is nondegenerate on $\hat M=B_{\Rk{\uk}}(\cvertex)\subset T_pM$ and we can pull back the  Riemannian metric from $g$ from $M$ to $\hat M$. From now on we will consider $\hat M$ with the pullback metric $\hat g=\exp^*(g)$. In order to avoid confusion with the exponential map of $\hat M$ we will denote the projection $\exp\co \hat M\to M$ by $\pi$.

Since in case $\uk\le 0$ we have that $\Rk{\uk}=\infty$ and the map $\exp\co T_pM\to M$ is an actual universal cover of $M$, we will be primarily interested in the case $\uk>0$.

By construction we have that the injectivity radius of $\cvertex$ in $\hat M$ is equal to $\Rk{\uk}$. If $R<\Rk{\uk}$ then the sectional curvature bound $\sec(\hat M)\le \uk$ implies that the boundary of the ball $\bar B_R(\cvertex)$ is locally convex. Thus $\bar B_R(\cvertex)$  is locally $\CAT(\uk)$.  Furthermore, Rauch comparison implies that the radial contraction towards $\cvertex$ is 1-Lipschitz on $\bar B_R(\cvertex)$. Hence any closed curve of length $\le 4\Rk{\uk} =\frac{2\pi}{\sqrt K}$ in $\bar B_R(\cvertex)$ can be contracted to a point through curves of length $\le 4\Rk{\uk}$. By the globalization theorem for $\CAT$ spaces ~\cite{AKP-CAT}  this implies that  $\bar B_R(\cvertex)$ is $\CAT(\uk)$ globally. In particular, any curve in $\bar B_R(\cvertex)$ is homotopic rel endpoints to the unique shortest geodesic connecting endpoints by a curve shortening homotopy.

Therefore
{ any two curves in $\bar B_R(\cvertex)$ of length $\le l<\Rk{\uk}$ with the same end points are homotopic rel endpoints through curves of length $\le l$.}

\begin{definition}
A homotopy $\gamma_t$ of loops at $p$ is called \emph{short} if all the loops $\gamma_t$ have length $< \Rk{\uk}$. Short homotopy is obviously an equivalence relation and equivalence classes are called \emph{short homotopy classes}.
\end{definition}
 We will call a loop at $p$ $l$-short if it has length $\le l$. We'll also call $ \Rk{\uk}$-short loops just short loops.
 Every short loop has a unique lift starting at $\cvertex\in \hat M$. More generally every short homotopy $\gamma_t$  of loops at $p$ uniquely lifts to a short homotopy $\tilde \gamma_t$ of paths starting at $\cvertex$. Discreteness of the fiber $\pi^{-1}(p)$ implies that the end points of $\tilde \gamma_t$ are the same. This gives a map $\phi$ from the set of  short homotopy classes to the fiber over $p$ given by $\phi([\gamma])=\tilde\gamma(1)$. This map is easily seen to be a bijection. Also, each short homotopy class $[\gamma]$ of short loops at $p$ contains a unique geodesic loop equal to the projection of the unique shortest geodesic from $\cvertex$ to $\phi([\gamma])$.
 
Given a short class $\alpha=[\gamma]$ we define its displacement $|\alpha|$ as $d(\cvertex, \phi([\gamma])|$. 
Given two short homotopy classes  $\alpha=[\gamma_1]$, $\beta=[\gamma_2]$ let $\gamma_1\cdot\gamma_2$ be the concatenation of $\gamma_1$ followed by $\gamma_2$.
It is obvious that
if  $|\alpha|+|\beta|<\Rk{\uk}$ then $\gamma_1\cdot\gamma_2$ is short. 
\begin{definition}

If $|\alpha|+|\beta|<\Rk{\uk}$   we define the \emph{Gromov product} $\beta\star \alpha$ as the short class $[\gamma_1\cdot\gamma_2]$.
\end{definition}

We can think of the Gromov product as an operation on short homotopy classes or on closed geodesic loops starting at $p$ or on elements of $\pi^{-1}(p)\cap B_{\Rk{\uk}}(\cvertex)$.

The Gromov product is associative i.e $(\alpha_1\star\alpha_2)\star\alpha_3=\alpha_1\star(\alpha_2\star\alpha_3)$  if $|\alpha_1|+|\alpha_1|+|\alpha_1|<\Rk{\uk}$ ~\cite[Proposition 2.2.5]{BuKa}.
It is also obvious that the constant short homotopy class $e$ satisfies $e\star\alpha=\alpha\star e=\alpha$ for any short class $\alpha$. Further if $\alpha=[\gamma]$ then the class of $\bar \gamma(t)=\gamma(1-t)$ is the inverse of $\alpha$.

This means that the set $\Gamma$ of all  short homotopy classes of loops at $p$ forms a pseudo-group  with respect to the Gromov product in the sense of the definition above. 

The pseudo-group $\Gamma$ pseudo-acts on $\tilde M$ \emph{on the the right} by the standard construction in covering space theory. This action defines a left action in the usual way by $\alpha(x):=x\cdot \alpha^{-1}$.

This left action  is the action we will work with from now on.
The action is only a pseudo-action because $\alpha\cdot x$ might not be defined for some pairs $\alpha\in\Gamma, x\in\tilde M$. However, when all the terms are defined the usual axioms of an action hold, that is

\[
e\cdot x=x,\quad  (\alpha\star\beta)(x)=  \alpha(\beta(x))
\]

It is easy to see that $\alpha(x)$ is always defined if $|\alpha|+d(x,\cvertex)<\Rk{\uk}$.
Further the pseudo-action is free, properly discontinuous, isometric and commutes with $\pi$. It is transitive on intersection of any fiber of $\pi$ with $B_{\Rk{\uk}/3}(\cvertex)$.

For any $0<R<\Rk{\uk}$ let $\Gamma(R)$ be the set of classes in $\Gamma$ with displacement $<R$.

Let us note here that any two loops of length $<R$ are short $R$-homotopic iff they are short homotopic. Therefore the pseudo-group $\Gamma(R)$ enjoys all the same properties with respect to $B_R(\cvertex)$ as the whole pseudo-group $\Gamma$ does with respect to $\hat M=B_{\Rk{\uk}}(\cvertex)$.

There is an obvious canonical homomorphism $\rho\co \Gamma(R)\to \pi_1(M,p)$ which maps the short homotopy class of a short loop to its homotopy class in  $\pi_1(M,p)$. The map $\rho$  induces a group homomorphism
 $\bar \rho\co \bar\Gamma(R)\to \pi_1(M,p)$. 
 In general   $\bar \rho$  need not be either 1-1 or onto. However if $M$ is compact then $\pi_1(M,p)$ is well known to be generated by loops of length $\le 2\diam M$.
 
 This immediately yields
 
 \begin{lemma}\label{lem:pseudo-onto}
Suppose  $\Rk{\uk}/2>R>2\diam M$. Then the canonical homomorphism $\bar \rho\co  \Gamma(R)\to \pi_1(M,p)$  is onto. 
  \end{lemma}

Furthermore, the following slightly  stronger assumptions  guarantee that $\bar\rho$ is  also 1-1 and hence an isomorphism.

 \begin{lemma}\label{lem:pseudo-iso}
 Suppose  $\Rk{\uk}/2>R>5\diam M$. Then the canonical homomorphism $\bar \rho\co  \Gamma(R)\to \pi_1(M,p)$  is   an isomorphism.
  \end{lemma}
\begin{proof}
A detailed proof of this lemma is given in \cite[Proposition 2.2.7]{BuKa}. We reproduce the proof here as we will need its details in the proof of the main theorem.

Let $W(R)=\langle a_g|g\in \Gamma(R)\rangle$ be the free group on elements of $\Gamma(R)$. We have a canonical homomorphism $\psi\co W(R)\to \pi_1(M,p)$ which is an epimorphism by the previous lemma. We need to show that its  kernel is exactly the normal subgroup $N_R$  generated by words of the form $a_{g_1}\cdot a_{g_2}\cdot a_{g_3}^{-1}$ whenever $g_1,g_2,g_3\in \Gamma(R)$ satisfy $g_1\star g_2=g_3$. The inclusion $N_R\subset \ker \psi$ is obvious and we just need to show that if $w\in \ker \psi$ then $w\in N_R$.

Let $c_1\co [0,1]\to M$ be a loop at $p$ representing $w$. It is nullhomotopic  by assumption. Let $c(s,t),0\le t,s\le 1$ be a piecewise smooth homotopy where $c_1(s)=c(s,1)$ and $c_0(s)=c(s,0)$ is the constant loop at $p$.

The curve $c_1$ comes with a subdivision $0=\sigma_0<\sigma_1<\ldots<\sigma_m=1$ where each $c_1|_{[\sigma_{i-1},\sigma_i]}$ corresponds to a letter $a_{\alpha_i}$ in the word $w$ for some $\alpha_i\in\Gamma(R)$.

Let us introduce a new subdivision point $s\in [\sigma_{i-1},\sigma_i]$.  We can connect $p$ to $c_1(s)$ by a shortest geodesic $g_s$. Then the closed loop $\alpha_i$ is homotopic to the product of the loop $\alpha_i'=c|_{[\sigma_{i-1},s]}\cdot \bar g_s$ followed by the loop $\alpha_i''=g_s\cdot c_1|_{[s,\sigma_i]}$. Since $g_s$ is a shortest geodesic both of these loops are $R$-short and $\alpha_i=\alpha_i'\star \alpha_i''$. Since this relation is in $N_R$ if we replace $a_{\alpha_i}$ by $a_{\alpha_i'}\cdot a_{\alpha_i''}$ the resulting word $w'$ will be equal to $w$ mod $N_R$.

By uniform continuity of the homotopy $c_t$ we can find subdivisions $0=s_0<s_1<\ldots<s_L=1$  which is  a refinement of $\{\sigma_i\}$ and $0=t_0<\ldots<t_e=1$ such that the curves given by restricting $c(s,t)$ to the sides of any rectangle $s_{i-1}\le s\le s_i, t_{j-1}\le t\le t_j$ are shorter than $\delta\ll\diam M$.

Let $g_{i,j}$ be a shortest geodesic from $p$ to $c(s_i,t_j)$. Obviously, it has length $\le\diam M$.

 Let $c_{ij}$ be the curve $s\mapsto c(s,t_{j})$ for $s\in [s_i,s_{i+1}]$ and $c^{ij}$ be the curve $t\mapsto c(s_i,t)$ for $t\in [t_j,t_{j+1}]$. Then each $c_{t_j}$ is the  product of curves $g_{ij}\cdot c_{ij}\cdot g_{i+1,j}^{-1}$ each of which has length $\le 2\diam M+\delta$. We will denote the short homotopy class of this curve by $\beta_{ij}$.
 
 Also let $\gamma_{ij}$ represent the short homotopy class of the similarly defined curve $g_{ij}\cdot c^{ij}\cdot g_{i,j+1}^{-1}$ which also has length $\le 2\diam M+\delta$ for the same reason. Then we obviously have that
 $\beta_{ij}\star \gamma_{i+1,j}=\gamma_{i,j}\star \beta_{i,j+1}$. These are short relations by the assumptions on the diameter of $M$ and therefore 
 \[
 a_{\beta_{ij}}\cdot a_{ \gamma_{i+1,j}}=a_{\gamma_{i,j}}\cdot a_{ \beta_{i,j+1}}\quad \text{ mod }N_R
 \]

Also, the classes $\beta_i,0$ and $\gamma_{0,j}, \gamma_{L,j}$ are all trivial because $c(s,0)=c(0,t)=c(1,t)=p$ for all $t,s$.

By induction on $i$ this easily gives that 
\[
a_{\gamma_{0,j}}\cdot a_{\beta_{0,j+1}}\cdot a_{\beta_{1,j+1}}\cdot \ldots\cdot a_{\beta_{L,j+1}}=a_{\beta_{0,j}}\cdot a_{\beta_{1,j}}\cdot \ldots\cdot a_{\beta_{L,j}}\cdot a_{\gamma_{L,j}}\quad \text{mod }N_R
\]

Since this is true for all $j=0,\ldots e-1$ we get that $w$ is equivalent to $a_{\beta_{0,0}}\cdot a_{\beta_{1,0}}\cdot\ldots a_{\beta_{L,0}}=1$ mod $N_R$.

\end{proof}

\begin{remark}\label{rem: non-injective}
While Lemma \ref{lem:pseudo-iso} implies that the map $\bar \Gamma(R)\to \pi_1(M)$ is 1-1 it 
  \emph{does not} automatically imply that the map $\Gamma(R)\to \pi_1(M)$ is 1-1 except in the case $\uk\le 0$. A counterexample immediately follows from a result of Buser and Gromoll ~\cite{Buser-Gromoll} originally claimed by Gromov~\cite{Gr-almost-flat} that for every $\eps>0$ there exists a Riemannian metric on $\SS^3$ with $\sec\le \eps$ and $\diam\le \eps$.
This non-injectivity creates considerable difficulty in the proof of the main theorem as well as in Gromov's proof of the original Almost Flat Manifolds Theorem.

\end{remark}

\section{Geometry of large mixed curvature almost flat manifolds}\label{sec:  RCD+CAT}
In this section we study the geometry of manifolds $(M^n,g,e^{-f}\cdot \mathcal H_n)$ with  $\sec\le \eps^2, \Ric_{f,N}\ge -(N-1)\eps^2$ with small $\eps$ and large injectivity radius and show that they are quantitatively close to $\R^n$. In particular we show that their large scale geometry is close to their infinitesimal geometry. This is a key point that serves as a substitute for the lower sectional curvature bound in the proof of the Fibration Theorem.

We will need the following definitions

 \begin{definition}
Given two complete Riemannian manifolds $M^n, N^n$, points $p\in M, q\in N$ and $R>10\eps>0$ we will say that $B_R(p)$ and $B_R(q)$ are $\eps$-Lipschitz-Hausdorff close if there exists an $\eps$-GH -approximation
$f\co B_R(p)\to B_{R+2\eps}(q)$ which is $(1+ \eps)$-Bilipschitz onto its image that is for any $x,y\in B_R(p)$ it holds that

\[
\frac{1}{1+\eps}d(x,y)\le d(f(x),f(y))\le (1+\eps)d(x,y)\quad 
\]
\end{definition}
\begin{remark}
Invariance of domain theorem easily implies that if $f\co B_R(p)\to B_{R+2\eps}(q)$  is a Lipschitz-Hausdorff approximation as above then $f(B_R(p))\supset B_{R-2\eps}(q)$.
\end{remark}

  \begin{definition}\label{def:almost-linear}
  We will say that a  smooth  map $f\co M^n\to B^m$ is $\eps$-linear on scale $\delta$  at $p\in M$ if 
     for any 
     $v\in T_pM, |v|\le 1$  it holds
    
    \[
         d(f(\exp_p(t v)),\exp_{f(p)}(td_pf(v))\le \eps\cdot\delta \quad \text{ for all } t\in [0,\delta]
    \]
    
    We'll say that $f$ is $\eps$-linear on scale $\delta$ on $U\subset M$ if it's $\eps$-linear on scale $\delta$ at any $p\in U$.
    \end{definition}
Note that we do not require any geodesics in the above definition to be minimizing. In particular a Riemannian covering between two complete manifolds is $\eps$-linear on scale $\delta$  for any positive $\eps$, $\delta$.

Also note that the almost linearity for a map $f\co M\to \R^m$ at $p\in M$ means that for any geodesic $\gamma$ with $\gamma(0)=p,\gamma'(0)=v\in T_xM, |v|=1$ it holds
    \[
       f(\gamma(t))=f(p)+td f_p(v)\pm\eps\cdot\delta\quad \text{ for all } t\in [0,\delta]
    \]

        \begin{remark}\label{rem: loc-linear}
    It is obvious from the triangle inequality that if $f\co M\to B_1$ is   2-Lipschitz and   $\eps$-linear on scale $\delta$ and $\Phi\co B_1\to B_2$ is $\vk(\delta)$-close to an isometry in $C^3$ then $\Phi\circ f$ is $(1+\vk(\delta))\cdot\eps$ linear on scale $\delta$. Therefore if $\delta$ is small, at the expense of slightly increasing $\eps$, for the purposes of verifying almost linearity we can work with $\Phi\circ f$ instead of $f$.

    In particular, observe that $(\frac{1}{\delta} B, q)$ converges  in $C^3$ to $(\R^m,0)$ as $\delta\to 0$  uniformly on compact subsets. Thus the above discussion applies where as a map $\Phi$ we can take a strainer map.  More precisely, 
for $R>0$ let $\Phi_R\co \R^m\to\R^m$ be the approximate Buseman map $\Phi_R(x)=(R-|x-Re_1|,\ldots, R-|x-Re_m|)$. It's easy to see that $\Phi_R\to \mathrm{Id}$ as $R\to \infty$ in $C^3$ uniformly on compact sets.
   Combining this with the previous remark means that again at the expense of slightly increasing $\eps$, to verify $\eps$-almost linearity on scale $\delta$ we can work with $\Phi\circ f$ where $\Phi$ is a strainer map near $q$ with respect to the strainer $(a_1=\exp_q(re_1),\ldots a_m=\exp_q(re_m))$ where $e_1,\ldots e_m$ is an orthonormal basis in $T_qB$ and $r>0$ is a small fixed constant $\ll \injrad(q)$ and $\delta$ is sufficiently small.

\end{remark}

We will need the following two results from the theory of $\RCD$ spaces.
\begin{theorem}[Splitting Theorem]\label{splitting-thm}\cite{giglisplitting, gigli-split-overview, Kap-Ket-18}
Let $1\le N<\infty$ and let $(X,d,m)$ be $\RCD(0,N)$ and $\CAT(0)$. Suppose $X$ contains a line. Then $(X, d)$ splits isometrically as $\R\times Y$.
\end{theorem}

The Splitting Theorem holds in general without the $\CAT(0)$ assumption. This was proved by Gigli in \cite{giglisplitting}( see also \cite{gigli-split-overview}) which generalizes the Splitting Theorem for Ricci limits proved by Cheeger and Colding \cite{CC-splitting}. Gigli also showed that the measure naturally splits as well. However, we only need the theorem in the special case above. In this case a simple
proof was given in  ~\cite{Kap-Ket-18} under the assumption that $X$ is $\CD(0,N)$ (rather than $\RCD(0,N)$) and $\CAT(0)$.

We will also need the following result on lower semicontinuity of geometric dimension for $\RCD$ spaces. For simplicity we only state it for smooth spaces as this is the only version we will need.

\begin{theorem}[Lower Semicontinuity of geometric dimension]\cite{KitaPOTA}
Let $1\le N<\infty$ and let $(M^n_i,g,e^{-f_i}\cdot \mathcal H_n, p_i)$ be a sequence of smooth weighted complete Riemannian manifolds satisfying $\Ric_{f_i,N}\ge \lk$ which measured pGH converges to $(N^m,g,\mu, p)$ where $(N^m,g)$ is a smooth Riemannian manifold. Then $m\le n$.
\end{theorem}

\begin{lemma}\label{lem-bilip-almost-split}
For any $N>1$ there are  universal functions  $\tilde R(\eps)=\Vk(\eps|N),  \tilde r(\eps)= \Vk(\eps| N), \nu(\eps)=\vk(\eps|N)$ with $\frac{1}{\sqrt\eps}>\tilde R(\eps)\ge \tilde r(\eps)$ such that the following holds.

Suppose $(M^n,g,e^{-f}\cdot \mathcal H_n, p)$ is a (not necessarily complete) weighted Riemannian manifold satisfying $\sec\le \eps^2, \Ric_{f,N}\ge -(N-1)\eps^2$ and the injectivity radius at $p$ is $\ge \frac  1\eps$ and the ball $\bar B_{1/\eps}(p)$ is compact.

Then

\begin{enumerate}
\item \label{large-ball-GH} The ball $B_{\tilde R}(p)$ is $\nu(\eps)$-GH close to  $B_{\tilde R}(0)\subset \R^m$
\item \label{comp-close-angle}
 For any $q\in B_{\tilde r}(p)$ and any $x,y\in B_{\tilde R}(p)$ it holds that 
\begin{equation*}\label{comp-close-angle}
\text{the comparison angle  $\tilde\sphericalangle xqy$ is $\nu(\eps)$-close to the actual angle $\angle xqy$}
\end{equation*}
\item  \label{bilip+GH} for any $1\le r<\tilde r$ the ball $B_r(p)$ is  $\nu(\eps)$-Lipschitz-Hausdorff close to  $B_r(0)\subset \R^n$.
\end{enumerate}

\end{lemma}
\begin{proof}
By the same argument as in the construction of pseudo-covers in Section~\ref{sec:pseudo-covers} the closed ball $\bar B_{1/100\eps}(p)$ is $CAT(\eps^2)$.  Since it has convex boundary and satisfies  $\Ric_{f,N}\ge -(N-1)\eps^2$ it is $\RCD( -(N-1)\eps^2, N)$. 
By the Splitting Theorem this implies that $(\bar B_{1/100\eps}(p), p)\to (\R^k,0)$ in pGH topology as $\eps\to 0$ for some $k\le N$. By semi-continuity of geometric dimension $k\le n$. 

 On the other hand, let $e_1,\ldots, e_n$ be an orthonormal basis in $T_pM$. Consider the points $a_i^\pm=\exp_p(\pm e_i)$. Then the $\CAT(\eps^2)$ condition implies that these points form an $n$-strainer at $p$ of size $1$. In other words $d(a_i^+,a_i^-)=2, d(a_i^\pm, p)=1$ for all $i$ and for all $i\ne j$ it holds that $d(a_i^\pm,a_j^\pm)\ge \sqrt{2}-\vk(\eps)$. Since no $\vk(\eps)$-close to $\{p, a_i^\pm\}$ configuration of points  exists in $\R^k$ for $k<n$ we must have that $k\ge n$.

Therefore $k=n$ and $X=\R^n$. 
\footnote{the above argument is the \emph{only} place in the proofs of the Fibration Theorem and of the Asphericity Theorem where the lower Bakry\textendash Emery bound  is used.}

Pointed convergence to $\R^n$ implies that there exists  $\frac{1}{\sqrt\eps}>\tilde R(\eps)\to\infty$ as $\eps\to 0$ such that  $B_{\tilde R(\eps)}(p)$ is $\vk(\eps|N)$-mGH close to $B_{\tilde R(\eps)}(0)\subset \R^n$.


Now let $1<r<\tilde R$. Then the logarithm map $\Psi_r\co \frac{1}{r}S_r(p)\to \Sigma_pM\cong\SS^{n-1}$ is 1-Lipschitz and onto. Since  $\frac{1}{r}S_r(p)$ is $\vk(\eps)$-close to $\SS^{n-1}$  this in turn implies that $\Psi_r$ is a  $\vk(\eps)$-GH-approximation. This implies that for any $x,y\in B_r(p)$  the comparison angle $\tilde\sphericalangle xpy$ is $\vk(\eps)$-close to the actual angle $\angle xpy$ between the geodesics $[px]$ and $[py]$. 
Moreover if we take $0<\tilde r(\eps)<\tilde R(\eps)$ such that $\tilde r(\eps)\to\infty$ but $\frac{\tilde r(\eps)}{\tilde R(\eps)}\to 0$ (e.g. $\tilde r=\sqrt{\tilde R}$ will do) then
the same argument applies to any $q\in B_{ {\tilde r}}(p)$ and hence \eqref{comp-close-angle} holds for any such $q$ and any $x,y\in B_{\tilde R}(p)$.

 
 This implies the following. Let $a_1,\ldots,a_n,b_1,\ldots,b_n$ be points in $B_{\tilde R}(p)$ which are  $\vk(\eps)$-close to \\
 $\tilde Re_1,\ldots,\tilde Re_n,-\tilde Re_1,\ldots,-\tilde Re_n$ in $\R^n$. Consider the strainer map $\Phi$ on  $M$ given by $\Phi(x)=(d(x,a_1),\ldots, d(x,a_n))$.  
Then  \eqref{comp-close-angle} and the first variation formula imply that for any $0<r<\tilde r(\eps)$ the gradients of the components of $\Phi$ given by $\nabla \Phi_1,\ldots \nabla \Phi_n$ are $\vk(\eps)$- almost  orthonormal on $B_r(p)$.  By possibly making $\tilde r(\eps)$ smaller but still very slowly going to $\infty$ as $\eps\to 0$ this gives that $\Phi$ is locally $(1+\vk(\eps))$-biLipschitz. This means that the fibers of $\Phi$ are discrete. We claim that the are actually points. Let $x_0\in B_r(p)$ and let $z_0=\Phi(x_0)$.  Any curve $\gamma$ starting at $z_0$ of length $l\le r$ can be uniquely  lifted to a curve $\tilde\gamma$  starting at $x_0$ of length $\le (1+\vk(\eps)l$. 
Let $U_l=\Phi^{-1}(B_l(z_0))$. Then by above we have that $ B_{(1+\vk(\eps))l/2}(x_0)\subset U_l\subset B_{(1+\vk(\eps))l}(x_0)$.
Consider the contraction of the ball $B_1(z_0)$ to $z_0$ along radial lines. This contraction lifts to a deformation retraction of $U_1$ to the fiber $\Phi^{-1}(z_0)$. Therefore  $U_1$ has the same number of connected components as $\Phi^{-1}(z_0)$. On the other hand  $\diam \Phi^{-1}(z_0)\le \vk(\eps)\ll 1$ and  hence $\Phi^{-1}(z_0)\subset B_{1/2}(x_0)\subset U_1$. Since $B_{1/2}(x_0)$ is path connected this implies that $U_1$ and hence $\Phi^{-1}(z_0)$ are path connected too. Hence $\Phi^{-1}(z_0)$ is a point.

Therefore $\Phi$ is 1-1 on $B_r(p)$ and thus it's globally and not just locally Bilipschitz there. This easily implies that
$\Phi$ satisfies \eqref{bilip+GH}.

By setting $\nu(\eps)$ to be the maximum of various $\vk(\eps|N)$ quantities appearing above we obtain the result.

\end{proof}

In the above proof let $A_i\subset B_{10\sqrt{\tilde R}}(a_i),  i=1,\ldots n$ and let us change $\Phi$ to \\ $\Phi(x)=(d(x,A_1),\ldots, d(x,A_1))$.
The same proof shows that thus redefined $\Phi$ still gives a $\vk(\eps)$-Lipschitz-Hausdorff approximation from $B_r(p)$ to the corresponding ball in $\R^n$.
Then we claim that the following holds

\begin{lemma}\label{lem-almost-linear}
The map $\Phi$  satisfies the following 

Let $x\in B_1(p)$ and  $v\in T_xM$ be unit length. Let $\gamma(t)$ be the geodesic with initial vector $v$. Then for any $t\in [0,1]$ we have 
\begin{equation}\label{eq-alm-lin-1}
\Phi(\gamma(t))'=\Phi(\gamma(t))'|_{t=0}\pm\vk(\eps|N)
\end{equation}
 and in particular $\Phi$  is $\vk(\eps)=\vk(\eps|N)$- linear on $B_{10}(p)$ on scale 1. I.e it satisfies
\begin{equation}\label{eq-alm-lin-2}
\Phi(\gamma(t))=\Phi(x)+tD_v\Phi(x)\pm\vk(\eps)
\end{equation}
were $D_v\Phi=d_x\Phi(v)$ is the directional derivative of $\Phi$ at $q$ in the direction $v$ given by the first variation formula.
\end{lemma}
\begin{proof}
 Let $h\co B_{\tilde R}(p)\to B_{\tilde R}(0)\subset \R^n$ be a $\vk(\eps)$-approximation. Let $\bar A_i=h(A_i)$. Let $\bar \Phi(x)=(d(x, \bar A_1),\ldots, d(x, \bar A_n))$  be the corresponding strainer map in $\R^n$.
Then $|\bar \Phi\circ h-\Phi|\le \vk(\eps)$ on $ B_{\tilde R}(p)$. Let $y=\gamma(1)$, $\bar x=h(x), \bar y= h(y)$ and let $\bar \gamma(t)=\bar x+t(\bar y-\bar x), t\in [0,1]$ be the constant speed geodesic from $\bar x$ to $\bar y$.
From the first variation formula it's easy to see that the lemma holds for $\bar \Phi$ and therefore

\[
\bar \Phi(\bar \gamma(t))=\bar \Phi(\bar x)+tD_{\bar v}\bar \Phi(\bar x)\pm\vk(\eps) \quad \text{ for } t\in [0,1]
\]

where $\bar v=\bar y-\bar x$.

Since $|h(\gamma(t))- \bar \gamma(0)|=t\pm \vk(\eps)$,  $|h(\gamma(t))- \bar \gamma(1)|=(1-t)|\pm \vk(\eps)$ and $ |\bar \gamma(0)- \bar \gamma(1)|=1\pm\vk(\eps)$  we have that $|h(\gamma(t))-\bar \gamma(t)|\le \vk(\eps)$ for all $t\in [0,1]$. Therefore $|\Phi(\gamma(t))-\bar \Phi(\bar \gamma(t))|\le \vk(\eps)$ for all $t\in [0,1]$. Combining this with the above we get that

\[
 \Phi( \gamma(t))=\Phi(x)+tD_{\bar v}\bar \Phi(\bar x)\pm\vk(\eps) \quad \text{ for } t\in [0,1]
\]

Lastly, by the first variation formula and property \eqref{comp-close-angle} from Lemma~\ref{lem-bilip-almost-split} we get that 
\[
|D_{\bar v}\bar \Phi(\bar x)-D_v\Phi(x)|\le \vk(\eps)
\]

Combining the last two displayed formulas we get \eqref{eq-alm-lin-2}.  Since the same result applies to $x=\gamma(t)$ for any $t\in [0,1]$ we get \eqref{eq-alm-lin-1}.

\end{proof}
Note that the map $\Phi$ in the above lemma is not smooth in general. However, by the first variation formula it has a well defined differential at every point $x$ in $B_{\tilde R/2}(p)$ which is 1-homogeneous i.e. satisfies $d_x\Phi(\lambda v)=\lambda d_x\Phi(v)$ for any $v\in T_x$ and any $\lambda>0$.  In particular, this works for maps which are $DC$ (this includes $\Phi$)  i.e. in local coordinates have components which are locally differences of Lipschitz semiconcave functions.

We can extend the notion of an almost Riemannian submersion to maps having this property as follows

\begin{definition}
Let  $f\co M^n\to N^m$ be a Lipschitz map between two Riemannian manifolds such that for any $x\in M$ it has a well defined 1-homogeneous directional derivative map $d_xf\co T_xM\to T_{f(x)}N$.
We will say that $f$ is an $\eps$-almost Riemannian submersion if for any $x\in M$ the differential $d_xf$ restricted to the unit ball around the origin in $T_xM$ is uniformly $\eps$-close to an orthogonal projection.
\end{definition}
With this definition it's obvious that for any $k\le n$ the partial strainer map\\ $\Pi_k=(\Phi_1,\ldots,\Phi_k)\co B_{10}(p)\to \R^k$ in  Lemma~\ref{lem-almost-linear} is a  $\vk(\eps)$-almost Riemannian submersion. Moreover the intrinsic metric on its fibers is Lipschitz close to the extrinsic metric. Combined with an obvious rescaling argument the above gives

\begin{corollary}\label{cor-almost-riem-subm}
Under the assumptions  of Lemma~\ref{lem-almost-linear} for any $k\le n$ and any fixed $R>0$ for all small $\eps$ it holds that  the partial strainer map $\Pi_k=(\Phi_1,\ldots,\Phi_k)\co B_{R}(p)\to \R^k$ is a  $\vk(\eps|R)$-almost Riemannian submersion and
 $\Pi_k^{-1}(0)\cap B_{2R}(p)$ is $\vk(\eps|R)$-Lipschitz-Hausdorff close to $B_{2R}(0)\subset \R^{n-k}$.

\end{corollary}
\begin{remark}\label{rem-pseudo-1}
It is obvious that  any two curves in $\R^{n-k}$ of length $\le l$ with the same end points are homotopic  rel endpoints through curves of length $\le l$. 
Therefore in the above Corollary  when $\eps$ is small enough so that $\vk(\eps|R)<1/10$  any two curves of length $\le l\le R$ in  $\Pi_k^{-1}(0)\cap B_{R}(p)$ with the same end points are homotopic in $\Pi_k^{-1}(0)\cap B_{2R}(p)$ rel endpoints through curves of length $\le 2l$.
\end{remark}

\section{Fibration theorem}\label{sec:fibration thm}

In this section we prove the Fibration Theorem. We will prove the following pointed version of the Fibration Theorem which allows for non-compact limits

\begin{theorem}[Fibration theorem]\label{pointed fibration-thm}

Let $( M_i^n, p_i)\to (B^m, p)$ in  pointed Gromov\textendash Hausdorff topology where $B^m$ is a smooth Riemannian manifold  and all $(M_i, g_i, e^{-f_i}\cdot \mathcal H_n)$ satisfy \eqref{eq:cd+cba}. 

Fix  $R>0$  and let $h_i\co B_R(p_i)\to B_R(p)$ be $\eps_i$-GH-approximations with $\eps_i\to 0$ as $i\to \infty$.

Then there are  $\nu_i\to 0$
and  smooth maps $\pi_i\co B_R(p_i)\to B$  such that for all large $i$ it holds that 

\begin{enumerate}
\item\label{fibr-thm-gh-approx}  $\pi_i$ is $\nu_i$-uniformly close to $h_i$.
\item\label{fibr-thm-almost-rieman} $\pi_i$ is a  $\nu_i$-almost Riemannian submersion.

\item\label{fibr-thm-intrinsic-fibers} The intrinsic metrics on the fibers $\pi_i^{-1}(x)$ is $(1+\nu_i)$-Bilipschitz to the extrinsic metrics for all $x\in B_{R/2}(p)$.
\end{enumerate}

\end{theorem}

We first show that the Fibration Theorem holds locally on the base.

\begin{proposition}[Local Fibration Theorem]\label{local fibr thm}
 Let $( M_i^n,g_i,e^{-f_i}\cdot \mathcal H_n, p_i)$ satisfy  \eqref{eq:cd+cba}.
Suppose  $( M_i^n, p_i)\to (B^m, p_\infty)$ in  pointed Gromov\textendash Hausdorff topology where $B^m$ is smooth.\\  Let $d_{GH}(B_1(p_\infty), B_1(p_i))=\eps_i\to 0$ and let $h_i\co  B_1(p_i))\to  B_1(p_i)$ be an $\eps_i$-GH-approximation.

Then there is $\delta_0=\delta_0(p_\infty)>0$ such that for  any fixed  $0<\delta<\delta_0$  there exist smooth maps $\pi_i\co B_\delta(p_i)\to B$ which for all large $i$  satisfy

\begin{enumerate}
\item \label{gh-approx-close} $\pi_i$ is $\vk(\eps_i)$-close to the original GH-approximation $h_i$;
\item \label{alm-reim-1}$\pi_i$ is a  $\vk(\delta|N)$-almost Riemannian submersion;
\item \label{alm-linear-1}$\pi_i$ is   $\vk(\delta|N)$-linear on scale $\delta$;
\item \label{bilip-intrinsic} The intrinsic metric on the fibers  $\pi_i^{-1}(y)$ for $y\in B_{\delta/2}(p_\infty)$ is $(1+\vk(\delta))$-Bilipschitz to the extrinsic metric.
\end{enumerate}

\end{proposition}
For the proof we will need the following elementary lemma in Euclidean geometry which we will leave to the reader as an exercise. 

\begin{lemma}\label{lem:euclid-est}There is universal $0<\eps_0<1$ such that the following holds.

Let  $n\ge 1$ be an integer and let $A,B$ be subsets in $\R^n$ satisfying  $e_1\in A, -e_1\in B$ and $d(A,B)=2$.

Then $\diam (A\cap B_{1+\eps}(0))\le 5\sqrt{\eps}$ and $\diam (B\cap B_{1+\eps}(0))\le  5\sqrt{\eps}$ for all $\eps<\eps_0$.

\end{lemma}

\begin{proof}[Proof of the Local Fibration theorem]
By rescaling  we can assume that $\sec(M_i)\le 1, \Ric_{f_i,N}(M_i)\ge -(N-1)$ and $|\sec (B_1(p_\infty))|\le 1$ and $\injrad(q)\ge 1$ for any $q\in B_1(p_\infty)$.

 Let $R_\delta=\tilde R(\delta)$ where $\tilde R(\eps)$ the function  provided by Lemma~\ref{lem-bilip-almost-split}.

The curvature bound $|\sec_B|\le 1$ implies $\Ric_B\ge -(N-1)$. Therefore by  Lemma~\ref{lem-bilip-almost-split}  there is $\delta_0$ such that for all  $0<\delta<\delta_0$ we have

\begin{equation}\label{eq:alm-eucl-ball-1}
\frac{1}{\delta}B_{\delta\cdot R_\delta}(p_\infty) \text{ is } \vk(\delta)\text{-GH close to  }B_{R_\delta}(0)\subset \R^m
\end{equation}

Observe that by construction of $\R_\delta$ we have that $\delta\cdot R_\delta\le \sqrt\delta$.

Next, it is obvious that if we rescale the sequence and the limit by $1/\delta$ we still have the convergence $( \frac{1}{\delta}M_i, p_i)\to ( \frac{1}{\delta} B^m, p_\infty)$. More precisely, if $\eps_i(r)=d_{GH}(B_r(p), B_r(p_i))\to 0$ as $i\to \infty$ then
\begin{equation}\label{eq:large i}
d_{GH}(\frac{1}{\delta}B_r(p_\infty),\frac{1}{\delta} B_r(p_i))=\frac{\eps_i(r)}{\delta}=\vk(i|\delta, r)
\end{equation}
Let us denote $\frac{1}{\delta}M_i$ by $\Mid$ and $\frac{1}{\delta}B$ by $B_\delta$.

 Note that $( \Mid,e^{-f_i}\cdot \mathcal H_n)$ has $\sec\le \delta^2$ and satisfies $\Ric_{f_i,N}\ge -\delta^2$.
 
 Combining \eqref{eq:alm-eucl-ball-1} and \eqref{eq:large i} we get that

\begin{equation}\label{eq:large i-2}
d_{GH}(B_{R_\delta}^{\Mid}(p_i),B^{\R^m}_{R_\delta}(0))=\frac{\eps_i}{\delta}+\vk(\delta)=\vk(\eps_i|\delta)+\vk(\delta)
\end{equation}

 Let $a_1^\infty,\ldots, a_m^\infty, b_1,\ldots b_m^\infty$ be points in $B_{R_\delta}^{B_\delta}(p_\infty)$ such that  they  are $\vk(\delta)$-close to\\ $R_\delta e_1,\ldots,R_\delta  e_m,-R_\delta e_1,\ldots,-R_\delta e_m$ in $\R^m$. Lift them to $\vk(\eps_i|\delta)$-close points $a_1^i,\ldots, a_m^i, b_1^i,\ldots b_m^i$ in $\Mid$. Let $\Phi_\infty$ and $\Phi_i$ be the corresponding strainer maps on $B_\delta$ and $\Mid$ given by $\Phi_\infty(x)=(d(x,a_1^\infty ),\ldots, d(x,a_m^\infty))$ and  $\Phi_i(x)=(d(x,a_1^i),\ldots, d(x,a_m^i))$. Here distances are measured in $\Mid$ and $B_\delta$.
Then $\Phi_\infty$ is a local diffeomorphism near $p$. Set $\pi_i\co \Mid \to B_\delta$ given by $\pi_i=\Phi_\infty^{-1}\circ \Phi_i$.

{\bf Claim:} The maps $\pi_i$ satisfy the conclusions of the Proposition.

First, it is easy to see that since $\Phi_\infty$ is $(1+\delta)$-Bilipschitz on $B_{10}(p_\infty)$ in $B_\delta$ , the map $\pi_i$ is $\vk(\eps_i)$-close to $h_i$ for large $i$. This verifies part  \eqref{gh-approx-close} of th Local Fibration theorem.

The main issue is to verify the almost linearity and almost Riemannian submersion properties.

Obviously $B_{1}^{\Mid}(p_i)=B_\delta^{M_i}(p_i)$.
It is also obvious that being an $\eps$-almost Riemannian submersion is invariant under simultaneous rescaling of the domain and the target by the same constant. Therefore, if we show that  $\pi_i\co \Mid \to B_\delta$ is a   $\vk(\delta)$-almost Riemannian submersion then the same holds for $\pi_i$ viewed as a map from $M_i$ to $B$.

Also,  $\pi_i\co \Mid \to B_\delta$ is $\eps$-almost linear on scale 1 iff $\pi_i$ viewed as a map from $M_i$ to $B$ is $\eps$-almost linear on scale $\delta$.

 Next note that  since $\Phi_{\infty}$ is a $\vk(\delta)$-isometry on $B_1^{B_\delta}(p_\infty)$ which is $\vk(\delta)$-almost linear on scale $1$, it's enough to show that $\Phi_{i}$ is a $\vk(\delta)$-almost Riemannian submersion which is $\vk(\delta)$-almost linear on scale  1 on the unit ball around $p_i$ in $ \Mid$.
 
 Let us  fix a small $\delta$ such that $\vk(\delta)\ll 1$ and  an $i$ large enough so that $\frac{\eps_i}{\delta}\ll\delta$.  From now on we will work with $\Mid$ and all distances will be measured in that space and not in $M_i$. Since  $i$ and $\delta$ will be fixed we will drop the $i,\delta$ subindices and 
 work with $M=M_{i,\delta}=\frac{1}{\delta} M_i$ and $\Phi=\Phi_i$.

By the choice of $\delta $ and $i$ and  \eqref{eq:large i-2} we have that

\begin{equation}\label{eq:large i-3}
d_{GH}(B_{R_\delta}(p),B^{\R^m}_{R_\delta}(0))\le \vk(\delta)
\end{equation}
 
We have a strainer map $\Phi\co M\to \R^m$ given by
 $\Phi(x)=(d(x,a_1),\ldots, d(x, a_m))$ where $B_{R_\delta(p)}$ is $\vk(\delta)$-close to $B_{R_\delta}(p)\subset \R^m$ and $(a_1,\ldots, a_m,b_1,\ldots, b_m)$ is a strainer in  $B_{R_\delta}(p)$ which is $\vk(\delta)$-close to $R_\delta e_1,\ldots,R_\delta  e_m,-R_\delta e_1,\ldots,-R_\delta e_m$ in $\R^m$. Also, $\sec M\le \delta^2$ and $(M,e^{-f}\mathcal H_n)$ satisfies $\Ric_{f,N}\ge -\delta^2$.

 We aim to prove that $\Phi$ is a $\vk(\delta)$-almost Riemannian submersion on $B_1(p)$ which is almost linear on scale 1.

First let us "lift" $\Phi$ to the pseudo-cover at $p$.

Consider the map $\exp=\exp_{p}\co T_pM\to M$. Since $\sec M\le \delta^2$ we have that the conjugate radius of $ M$ is at least $\frac{\pi}{\delta}$. Recall that $R_\delta<\frac{2}{\sqrt \delta}<\frac{\pi}{20\delta}$.

 Pull back the Riemannian metric via  $\exp$ to the ball of $B_{R_\delta}(\cvertex)$ in $T_{p}M$.  We will denote the resulting Riemannian manifold by $\hat M$.  Here $\cvertex$ denotes the origin in $T_{p}M$.
 


 Let $A_j=\exp^{-1}(a_j)\cap B_{R_\delta+1}(\cvertex)$ and $B_j=\exp^{-1}(b_j)\cap B_{R_\delta+1}(\cvertex)$ for $j=1,\ldots, m$. 
  Let $\hat\Phi\co  \hat M \to\R^m$ be given by $\hat\Phi(x)=(d(x,A_1),\ldots, d(x,A_m))$.

Recall that $d(a_{j_1},a_{j_2})= \sqrt{2}R_\delta\pm\vk(\delta)$,  $d(b_{j_1},b_{j_2})= \sqrt{2}R_\delta\pm\vk(\delta)$, $d(a_{j_1},b_{j_2})= \sqrt{2}R_\delta\pm\vk(\delta)$  for $j_1\ne j_2$ and all large $i$. Also  $d(a_{j},b_{j})= 2R_\delta\pm\vk(\delta)$, $d(a_{j},p)= R_\delta\pm\vk(\delta)$ and $d(b_{j},p)= R_\delta\pm\vk(\delta)$ for all $j$. 

The local submetry property of $\exp$ implies that  $d(A_{j_1},A_{j_2})\ge \sqrt{2}R_\delta-\vk(\delta)$,  $d(B_{j_1},B_{j_2})\ge \sqrt{2}R_\delta-\vk(\delta)$, $d(A_{j_1},B_{j_2})\ge \sqrt{2}R_\delta-\vk(\delta)$  for $j_1\ne j_2$. Also  $d(A_{j},B_{j})= 2R_\delta\pm\vk(\delta)$, $d(A_{j},p)= R_\delta\pm\vk(\delta)$ and $d(B_{j},p_i)= R_\delta\pm \vk(\delta)$ for all $j$. Since by Lemma~\ref{lem-bilip-almost-split}  the ball $B_{R_\delta}(\cvertex)$ is $\vk(\delta)$-close to the ball of the same radius  in $\R^n$,  Lemma~\ref{lem:euclid-est} implies that 

\begin{equation}\label{eq:est-sq-root}
\diam A_j\le 10\sqrt{R_\delta}, \diam B_j\le 10\sqrt{R_\delta} \quad \text{ for all } j=1,\ldots, m
\end{equation}

Furthermore,   the following estimates hold
\begin{align}
d(A_{j_1},A_{j_2})=\sqrt{2}R_\delta\pm\vk(\delta), d(B_{j_1},B_{j_2})= \sqrt{2}R_\delta\pm\vk(\delta),\label{eq-str-1} \\ 
d(A_{j_1},B_{j_2})= \sqrt{2}R_\delta\pm\vk(\delta) \text{ for } j_1\ne j_2, j_1, j_2=1,\ldots, m;\label{eq-str-2}\\
d(A_{j},B_{j})= 2R_\delta\pm\vk(\delta), d(A_{j},\cvertex)= R_\delta\pm\vk(\delta)\label{eq-str-3}\\
 \text{ and } d(B_{j},\cvertex)= R_\delta\pm \vk(\delta) \text{ for all} j=1,\ldots m.\label{eq-str-4}
 \end{align}

In other words the sets $A_i, B_i, i=1,\dots, m$ form a strainer of size $R_\delta$.

Next observe that since $\exp$ is a local submetry it  satisfies that $d(x, \exp^{-1}(a_j))=d(\exp(x), a_j)$ for $x\in B_{1}(\cvertex)$.
 Also, by construction,  if $x\in B_{1/10}(\cvertex)$ then $d(x, \exp^{-1}(a_j))=d(x,A_j)$ since for any $y$ in $B_{R_\delta+1}(\cvertex)^c\cap \exp^{-1}(a_j)$ and  $x\in B_{1/10}(\cvertex)$  we have $d(x,y)\ge R_\delta+0.9>d(\cvertex, A_j)+0.1>d(x,A_j)$. Altogether this means that for  $x\in B_{1/10}(\cvertex)$ it holds that $d(\exp(x), a_j)=d(x, \exp^{-1}(a_j))=d(x,A_j)$. The same holds for $b_j$ and $B_j$.

 Therefore we have that 
 \begin{equation}
  \hat \Phi=\Phi\circ \exp \text{ on }B_{1/10}(\cvertex).
 \end{equation}
 Therefore, since $\exp$ is a local isometry to prove that $\Phi$ is an almost linear  almost Riemannian submersion  it's enough to show that this holds for $ \hat \Phi$. But the latter immediately follows from  Lemma~\ref{lem-almost-linear}. Indeed since $A_j, B_j, j=1,\ldots m$ is a partial strainer  and  $B_{R_\delta}(\cvertex)$ is $\vk(\delta)$-close to the  ball of radius $R_\delta$ in $\R^n$ we can complete the strainer $A_j, B_j, j=1,\ldots m$ to a full strainer $A_j, B_j, j=1,\ldots n$ of length $R_\delta$ (we can take $A_j, B_j, j=m+1,\ldots, n$ to be points) where the formulas \eqref{eq-str-1}, \eqref{eq-str-2}, \eqref{eq-str-3}, \eqref{eq-str-4} hold for all $j=1,\ldots, n$. Let us extend the map $\hat\Phi \co \hat M\to \R^m$ by the same formula to $\hat \phi\co \hat M\to \R^n$ by setting $\phi_j(x)=d(x,A_j)$ for $j=1,\ldots n$. By constructions $\hat \phi_j=\hat \Phi_j$ for $j=1,\ldots m$, i.e. $\hat \Phi=pr_m\circ \hat\phi$ where $pr_m\co \R^n\to\R^m$ is the canonical coordinate projection onto the first factor in $\R^n=\R^m\times\R^{n-m}$.

 Therefore by Lemma~\ref{lem-almost-linear}  $\hat \Phi|_{B_{1/10}(\cvertex)}$ (and hence $ \Phi|_{B_{1/10}(p)}$) is a $\vk(\delta)$-almost Riemannian submersion which is almost $\vk(\delta)$-linear on scale $1/10$.
 This verifies parts \eqref{alm-reim-1} and \eqref{alm-linear-1} of the Local Fibration theorem.
 
Part \eqref{bilip-intrinsic} follows from Corollary~\ref{cor-almost-riem-subm}.

Lastly, note that the maps $\Phi_i$ need not be smooth but  we can easily smooth them out preserving their properties. For example we can change the functions $d(\cdot, a_j^i)$ in the definition of $\Phi_i$ to the functions of the form $x\mapsto \int_{B_\nu(a_j^i)}d(x,y)d\vol_y$ where $B_\nu(a_j^i)$ is a tiny ball around $a_j^i$.  This approximates $\Phi_i$ by a  $C^1$ map which satisfies all the properties of $\Phi_i$ claimed by the Local Fibration Theorem.

\end{proof}

 
 \begin{remark}\label{non-complete-loc-fibra}
 In the above proof we only deal with balls of bounded size in $M_i$ and $B$. Therefore the assumption that $M_i$ be complete can be weakened to require that the closed balls $\bar B_1(p_i)$ be compact.
 \end{remark}
 
 By rescaling the Local Fibration theorem easily yields
\begin{corollary}\label{cor-large-alm-linear}
 Let $( M_i^n,g_i,f_i, p_i)$ satisfy  $\sec \le \delta_i$  and  $\Ric_{f_i,N}\ge -\delta_i$ with $\delta_i\to 0$.
 
Suppose  $( M_i^n, p_i)\to (\R^m, 0)$ in  pointed Gromov\textendash Hausdorff topology. Then there exist smooth maps $\pi_i\co (M_i, p_i)\to (\R^m,0)$ such that the following holds.

Fix  $\eps, R>0 $. Let $h_i\co B_R(p_i)\to B_R(0)$ be $\mu_i$-GH approximations with $\mu_i\to 0$. Then there exist $\eps_i(R)\to 0$ as $i\to \infty$ such that the restrictions of $\pi_i$ to $B_R(p_i)$ satisfy
\begin{enumerate}
\item $\pi_i|_{B_R(p_i)}$ is $\eps_i$-almost Riemannian submersion 
\item  $\pi_i|_{B_R(p_i)}$ is $\eps_i$-close to the original GH-approximation $h_i$;
 \item  $\pi_i|_{B_R(p_i)}$ is $\eps$-linear on scale $R$.
 \item\label{intr-bilip} The intrinsic metric on the fiber  $F_i=\pi_i^{-1}(0)$  is $(1+\eps_i)$-Bilipschitz to the extrinsic metric.

 \end{enumerate}
\end{corollary}\label{cor-pseudo-cover}
Note that in the above corollary we are allowing $R$ to be large.

\begin{corollary}\label{cor-pseudocov-2}
Under the assumptions of Corollary~\ref{cor-large-alm-linear} let $\exp\co \hat M_i\to M$ be the pseudo-cover  as defined in Section~\ref{sec:pseudo-covers} and let $\Gamma_i$ be the corresponding pseudo-group. Let $\hat \pi_i=\pi_i\circ \exp$. 
Then for any fixed $R>0$ 
the group $\bar \Gamma_i(R)$ is isomorphic to $\pi_1(F_i)$ for all large $i$.

\end{corollary}

\begin{proof}
By the construction of the map $\pi_i$ in the proof of the Local Fibration Theorem, the map $\hat \pi_i$ is equal to the first $m$ components of a full strainer map $\Phi^i\co\hat M_i\to \R^n$ which when restricted to $B_R(\cvertex)$ gives an $\eps_i(R)$-Lipschitz-Hausdorff approximation to a ball of the same radius in $\R^n$ where $\eps_i\to 0$ as $i\to\infty$.
Note that for all large $i$ we have that $R<\Rk{\delta}/20$. Therefore we can talk about the pseudo-group $\Gamma_i(R)$.

By construction the action of $\Gamma_i(R)$ leaves $\hat F_i=\hat \pi_i^{-1}(0)\cap B_{2R}(\cvertex)=\exp^{-1}(F_i)\cap B_{2R}(\cvertex)$ invariant. Let $[\gamma]\in\Gamma_i(R)$ be a short homotopy class of a geodesic loop $\gamma$ at $p$ of length $\le R$. Let $\tilde \gamma$ be its lift starting at $\cvertex$. The (shortest) geodesic $\tilde\gamma$ has endpoints in $\hat F_i$ but is not necessarily contained in $\hat F_i$. However, since $\Phi^i$ is a Lipschitz\textendash Hausdorff approximation from $B_{2R}(\cvertex)$ to $B_{2R}(0)\subset \R^n$ we can obviously "push" the curve $\tilde \gamma$ into $\hat F_i$ via a homotopy rel endpoints to a $\eps_i$-close curve $\hat\gamma$ with length $\le (1+10\eps_i)L(\gamma)$. Projecting down to $M$ gives a loop in $F_i$ based at $p$. Moreover a short homotopy of $\tilde \gamma$ rel end points can be pushed into $\hat F_i$ as well with the same control on the length of curves in the homotopy as above.

This shows that the map $\gamma\mapsto \exp(\hat \gamma)$  induces a well defined map $\phi_i\co \Gamma_i(R)\to \pi_1(F_i,p)$ which is obviously a homomorphism.

We claim that this map is an isomorphism for all large $i$. Taking into account that  $\diam F_i\ll R$ for large $i$ this follows by straightforward modifications of the proofs of Lemma~\ref{lem:pseudo-onto} and Lemma~\ref{lem:pseudo-iso}.

\end{proof}
Next we show that if under the assumptions of the local fibration theorem we have two different local approximating fibrations satisfying the conclusion of the theorem then they must be $C^1$ close and not just $C^0$ close.
This will allow us to glue different local fibrations into one global fibration in the proof of the global Fibration Theorem.

\begin{lemma}\label{local-fibr-C1-close}
Under the assumptions of the Local Fibration theorem \ref{local fibr thm} suppose  that for a small $\delta$ and all large $i$ we have two different maps
$\pi_i, \pi_i'\co B_\delta(p_i)\to B$ such that
\begin{enumerate}
\item  Both $\pi_i$ and $ \pi_i'$ are $\vk(\delta|N)$-almost Riemannian submersions;
\item Both $\pi_i$ and $ \pi_i'$ are $\vk(\delta|N)$-linear on scale $\delta$,
\item $\pi_i$ and $ \pi_i'$ are uniformly $\mu_i$-close with $\mu_i\to 0$ as $i\to\infty$.
\end{enumerate}

Then $p_i$ and  $p_i'$ are  $\vk(\delta|N)$-close in $C^1$ for large $i$ which more precisely means the following. Let $\Phi\co B_{2\delta}(p_\infty)\to \R^m$ be a strainer map 
for the strainer $(a_1=\exp_q(re_1),\ldots a_m=\exp_q(re_m)$ where $e_1,\ldots e_m$ is an orthonormal basis in $T_qB$ and $r>0$ is a  fixed constant $\ll \injrad(q)$.

Then
for any $x\in B_1(p_1)$ the maps $d_x(\Phi\circ \pi_i), d_x(\Phi\circ \pi_i')\co T_xM_i\to\R^m$ are $\vk(\delta)$-close for all large $i$.
\end{lemma}
\begin{proof}
 
Let $\hat\pi_i=\Phi\circ \pi_i$ and $\hat\pi_i'=\Phi\circ \pi_i'$.
Let $x\in B_\delta(p_i)$ and let $q=\hat\pi_i'(x)$. Let $w$ be any  unit vector in $\R^m$. Let $v\in T_xM_i$ be a horizontal lift of $w$ with respect to  $\hat\pi_i$ so that $d_x\hat\pi_i(v)=w$. Then $|v|=1\pm\vk(\delta)$.
 Let $\gamma\co [0,\delta]\to M_i$ be a geodesic  with $v=\gamma'(0)$. 
 Then by almost linearity we have that $\delta=d(q,q+\delta w)=d(p,\gamma(\delta))\pm\vk(\delta)\cdot\delta$. Therefore the geodesic $\gamma$ is almost minimizing on scale $\delta$ in the sense that $\delta(1\pm\vk(\delta)) =L(\gamma)=d(\gamma(0), \gamma(\delta))\pm\delta \vk(\delta)$. In other words 
 
 \[
 1 \le \frac{L(\gamma)}{d(\gamma(0),\gamma(\delta))}\le 1+\vk(\delta)
 \]
 
 Conversely, the same argument shows that if $\gamma$ is a constant speed geodesic that satisfies the last property then $v=\gamma'(0)$ is almost horizontal for $\pi_i$. 
 
 Lifting $m$ orthonormal vectors $w_1,\ldots, w_m$ horizontally we see for dimension reasons that a vector in $T_xM_i$  is $\vk(\delta))$-horizontal for $\hat\pi_i$ iff it's almost minimizing on scale $\delta$.
The same argument applies to $\hat \pi_i'$. Therefore $\hat\pi_i$ and $\hat \pi_i'$ have $\vk(\delta)$-close vertical and horizontal spaces.

Thus, to prove that $d_x \hat \pi_i$ and $d_x \hat \pi_i'$ are close it's enough to check that they are close on horizontal spaces.
 
 Let $v\in T_xM_i$ be a unit $\vk(\delta)$-horizontal( for both $\pi_i$ and $\pi_i'$)  vector  so that for $\gamma(t)=\exp_x(tv)$ we have that $d(\gamma(\delta),x)=\delta(1\pm\vk (\delta))$.
 Then by almost linearity $\hat \pi_i(\gamma(\delta))$ is $\vk(\delta)\cdot\delta$-close to $q+\delta d_x\hat\pi_i(v)$ where as before $q=\hat\pi_i(x)$.
 
 But the same argument works for $\hat \pi_i'$. And  since $\hat \pi'_i$ is $\nu_i$-close to $\hat \pi_i$ we get that $q'=\hat \pi_i'(x)$ is  $\nu_i$-close to $q$ and  $\hat \pi_i'(\gamma(1))$ is $\nu_i$-close to $\hat\pi_i(\gamma(1))$ by the triangle inequality we get that $|d_x\hat\pi_i(v)-d_x\hat\pi_i'(v)|\le \vk(\delta)+4\frac{\mu_i}{\delta}\le  \vk(\delta)$ for large $i$.
 
 This proves that $d_x \hat \pi_i$ and $d_x \hat \pi_i'$  ore $\vk(\delta)$-close on the horizontal spaces and hence on the whole unit ball in $T_xM_i$. This finishes the proof of the Lemma.

\end{proof}

\begin{proof}[Proof of the Fibration Theorem ( Theorem \ref{pointed fibration-thm})]


  With the  Local Fibration Theorem and Lemma~\ref{local-fibr-C1-close} at our disposal the full Fibration Theorem follows by a standard gluing argument similar to the one in proof of the fibration theorem in ~\cite{CFG}.

  Fix an $R>0$ and look at a maximal $\frac \delta{300}$-separated net $p_{\infty,1}, \dots,  p_{\infty,l}$ in $B_{2R}(p_\infty)$ then the balls $\{B_j=B_{\frac\delta{100}}(p_{\infty,j})\}_{j=1,\ldots,l}$ cover  $B_{2R}(p_\infty)$. By Bishop-Gromov volume comparison this cover has intersection multiplicity $\le c(n)$ for all sufficiently small $\delta$. Look at the corresponding lifted balls  $B_j^i=B_{\frac\delta{100}}(p_{i,j})$ covering $B_{2R}(p_i)$. For each j we have constructed an $\vk(\delta)$-almost Riemannian submersion $\pi_{i,j}\co B_j^i\to B$ which is $\nu_i$-close to the original $GH$-approximation $h_i$  where  $\eps_i\le\nu_i\to 0$.
  Furthermore, by Lemma~\ref{local-fibr-C1-close} the maps $\pi_{i,j}$ are  $\vk(\delta)$-$C^1$-close to each other on overlaps for large $i$.
 
   Then we can glue them into a global map using one of the standard center of mass constructions.
  
  For example, let $\{\rho_i^j\}_{j=1,\ldots l}$ be a smooth partition of unity subordinate to the cover $\{ B_j^i\}_{j=1,\ldots l}$ such that all the $\rho_i^j$ are $C(n)/\delta$-Lipschitz (such a partition of unity is easy to construct).
  Then consider the map $\pi_i\co B_{2R}(p_i)\to N$ defined by $x\mapsto \argmin \sum_j\rho_j(x)d^2(\cdot ,\pi_{i,j}(x))$. Note that if $B=\R^m$ then $ \argmin \sum_j\rho_j(x)d^2(\cdot ,\pi_{i,j}(x))=\sum_j\rho_j(x)\pi_{i,j}(x)$. Since $\bar B_{3R}(p_\infty)$ is compact, for all small $\delta$ the balls $\frac 1 \delta B_{\delta}(p_{j,\infty})$ are uniformly $\vk(\delta)$-$C^4$-close to the unit ball in $\R^m$. Therefore, locally  in some strainer coordinates  we can view all the maps on the balls $B_j^i$ contained in a $\delta/10$-neighborhood of $p_j^i$ as maps into $\R^m$. Moreover, in these coordinates the map $h$ is $\vk(\delta)$-$C^4$-close to the map
   $\pi_{i,\delta}\co M_i\to \R^m$ given by the formula $x\mapsto \sum_j\rho_j(x)d^2(\cdot ,\pi_{i,j}(x))$. By an elementary product rule calculation the Lipschitz bounds on $\rho_i$ and $C^1$-closeness on $\pi_{i,j}$'s on $B_{\delta}(p_j^i)$ imply that  $\pi_i$ is $10\nu_i$ uniformly close to $h_i$ and $\vk(\delta)$-$C^1$-close to each $\pi_{i,j}$ (where both maps are defined).  In particular $\pi_i$ is a $\vk(\delta)$-almost Riemannian submersion for all large $i$.
  
 By a diagonal argument with $\delta_i\to 0$ very slowly this immediately yields part ~\eqref{fibr-thm-gh-approx} and part ~\eqref{fibr-thm-almost-rieman}  of the Fibration Theorem.
 The proof of the part \eqref{fibr-thm-intrinsic-fibers}  is exactly the same as the proof of the corresponding part of the Local Fibration Theorem. We can lift the maps $\pi_i$ to a pseudo-cover $\exp_{p_{i,j}}\hat M_i\to M_i$ to 
 $\hat \pi_i=\pi_i\circ \exp_{p_{i,j}}$. The lifted map $\hat \pi_i$ is $C^1$-close to the local fibration map $\hat \pi_{i,j}$ which by the proof of the Local Fibration Theorem can be complemented to a $(1+\vk(\delta_i))$-Bilipschitz map to $\R^n$ on a ball of size $\gg \eps_i$. Complementing the map  $\hat \pi_i$ by the same functions we still get a  $(1+\vk(\delta_i))$-Bilipschitz map to $\R^n$ which yields that the intrinsic metrics on the fibers of  $\hat \pi_i$ are  $(1+\vk(\delta_i))$-Bilipschitz to the extrinsic ones. This property obviously descends down into $M_i$ which yields part \eqref{fibr-thm-intrinsic-fibers} and finishes the proof of the Fibration Theorem.

  \end{proof}
  \begin{remark}
  In the proof of the Main Theorem we will only apply the Fibration Theorem in the situation where the base $N$ is flat and  isometric to $T^k\times \R^{m-k}$  and hence  locally isometric to $\R^m$. In this case in the above proof we can define
  $\Pi$ directly by the formula $\Pi(x)=\sum_j\rho_j(x)\pi_{i,j}(x)$ using any local isometry to $\R^m$. This formula makes sense since taking  center of mass in $\R^m$ commutes with isometries.
  \end{remark}
     
    \begin{remark}\label{non-complete-global-fibra}
Just as in the case of the Local Fibration Theorem (see remark ~\ref{non-complete-loc-fibra})  in the statement of the Fibration Theorem we can weaken the assumption of completness of $M_i$ to the assumption that for some $R_i\to\infty$ the closed balls $\bar B_{R_i}(p_i)$ are compact.
    \end{remark}
    The construction of the fibrations in the Fibration Theorem involves some flexibility. However, the homotopy type of the fibers is well defined as the following general Lemma shows.

    \begin{lemma}\label{lem-hom-fiber}
Let $\pi_i, \pi_i'\co  (M_i^n, p_i)\to (B^m, p)$ be a sequence of $\eps_i$-almost Riemannian submersions with diameters of fibers $\le \eps_i$ where $\eps_i\to 0$ as $i\to\infty$.  Then for all large $i$ the fibers $F_i=\pi_i^{-1}(p)$ and  $F_i'={\pi_i'}^{-1}(p)$ are homotopy equivalent via a homotopy equivalence $f_i\co F_i\to F_i'$ which homotopy commutes with the inclusions $j\co F_i'\hookrightarrow M_i$ and $j'\co F_i \hookrightarrow M_i$, i.e $j\sim j'\circ f_i$. Furthermore $F_i$ (and hence $F_i'$) are connected for all large $i$.


    \end{lemma}
    \begin{proof}

Fix a small $\delta$ much smaller than the injectivity radius of $B$ at $p$. Let $B_\delta=\pi_i^{-1}(B_\delta(p))$ and $B'_\delta={\pi_i'}^{-1}(B_\delta(p))$. Then horizontal lifts of radial contractions along geodesics starting from $p$ give a deformation retraction of $B_\delta$ onto $F_i$ and also onto $B_{\delta/2}$. Hence $F_i\hookrightarrow B_{\delta/2}\hookrightarrow B_\delta$ are homotopy equivalences. Similarly $F_i'\hookrightarrow B'_{3/4\delta}\hookrightarrow  B'_{3/2\delta}$ are also homotopy equivalences. Since for large $i$ we have that $B_{\delta/2}\subset B'_{3/4\delta}\subset B_\delta\subset  B'_{3/2\delta}$ we get that the inclusion $ B'_{3/4\delta}\subset B_\delta$ is both surjective and injective on homotopy groups. This yields the  first claim.

To see that $F_i$ must be connected observe that $B_\delta\supset B_{\delta/2}(p_i)$ for all large $i$. Since $B_{\delta/2}(p_i)$ is connected and $B_\delta$ deformation retracts onto $F_i\subset B_{\delta/2}(p_i)$ it follows that $B_\delta$ is connected and hence so is $F_i$.
    
    \end{proof}
\begin{remark}
By the above lemma and  Lemma~\ref{local-fibr-C1-close} the homotopy  type of the fibers produced by the Fibration Theorem is well defined. The diffeomorphism type is also well defined but unlike the statement about the homotopy type this follows not from the conclusion of the Fibration Theorem but rather from the construction of the maps $\pi_i$ in its proof. To see this observe that in the diagonal argument in the proof of the Fibration Theorem we take $\delta_i\to 0$ very slowly so that $\nu_i\ll \delta_i$.  
Therefore the glued maps $\pi_i$ satisfy an extra condition that they are $\eta_i$-linear on scale $\delta_i$ for some $\eta_i\to 0$. Any two such maps must be $C^1$-close by Lemma~\ref{local-fibr-C1-close}.  This easily implies that  any two different fibration maps $\pi_i$ and $\pi_i'$ produced by the construction in the proof of the Fibration Theorem have diffeomorphic fibers. E.g. because one can use the same gluing construction to glue them into a single fibration map which agrees with $\pi_i$ and $\pi_i'$ on some disjoint balls $B_{\delta_i}(p_i)$ and $B_{\delta_i}(p_i')$ respectively. The same argument 
 also shows that the fibers of $\pi_i$ are diffeomorphic to the fibers of the local fibrations $\pi_{i,j}$ in the proof of the Fibration Theorem.

\end{remark}

\section{Induction Theorem}\label{sec: Induction thm}

\begin{theorem}[Induction Theorem]\label{thm:induction}
Supposes $n\le N$ and $( M_i^n, p_i)\to (\R^k, 0)$ in  pointed Gromov\textendash Hausdorff topology 
where  $(M_i, g_i, e^{-f_i}\cdot \mathcal H_n)$  is complete and satisfies $\sec \le \eps_i\to 0$ and  $\Ric_{f_i,N}\ge -\eps_i$. Let $F_i=\pi_i^{-1}(0)$ be the fiber for the local fibration map $\pi_i\co M_i\to B_1(0)\subset \R^k$ provided by the Fibration Theorem. 
Assume  that the inclusion $F_i\subset M_i$ is a homotopy equivalence (this is unambiguous thanks to Lemma~\ref{lem-hom-fiber})

Then for all large $i$  $F_i$ is aspherical and the universal covers $(\tilde M_i, \tilde p_i)$ pGH converge to $(\R^n, 0)$ for any $\tilde p_i$ in the fiber over $p_i$.

\end{theorem}

\begin{proof}[Proof of the Induction Theorem (Theorem \eqref{thm:induction})]
The proof proceeds by reverse induction on $k$. When $k=n$ there is nothing to prove as in this case the fiber of the map $\pi_i\co M_i^n\to \R^n$ is a point since it's  connected by Lemma~\ref{lem-hom-fiber}. By assumption the inclusion of the fiber into $M_i$ is a homotopy equivalence which means that $M_i$ is contractible. In particular it's equal to its universal cover.

Induction step. Suppose   $( M_i^n, p_i)\to (\R^k, 0)$ is a contradicting sequence   satisfying the assumptions of the theorem (i.e. no subsequence of it satisfies the conclusion of the theorem) and the result is already known if the limit is $\R^m$ with $n\ge m>k$.
Let $\pi_i\co M_i\to \R^k$ be the fibration map over $B_1(0)\subset \R^k$ from the Local Fibration Theorem. By Corollary~\ref{cor-large-alm-linear} we can choose $\pi_i$ to be $\eps_i$-almost Riemannian submersion on $B_{r_i}(p_i)$ with $r_i\to\infty$ such that the intrinsic metric on the fibers over $0$ is $(1+\eps_i)$-Bilipschitz to the extrinsic metric. Let $F_i=\pi_i^{-1}(0)$ and let $\lambda_i=\diam(F_i)$. Recall that $F_i$ is connected by Lemma~\ref{lem-hom-fiber}.
Let $N_i=\frac{1}{\lambda_i}M_i$.

Let us lift $\pi_i$ to the pseudo-cover $\exp\co \hat N_i\to N_i$ at $p_i$.  Let $\Gamma_i$ be the corresponding pseudo-group.
Then by Corollary~\ref{cor-pseudocov-2},  for any fixed $R>10$ we have  that $\bar \Gamma_i(R)\cong \pi_1(F_i)$ for all large $i$.

By precompactness we have that 
$(N_i,p_i)$ subconverges to $(Y,p_\infty)$ and $( \hat N_i, \Gamma_i)$ equivariantly subconverges to $(Z,G)$ such that $Z/G=Y$.  By Lemma~\ref{lem-bilip-almost-split} $Z=\R^n$. Note that since $\eps_i\to 0$ we have that any $\Rk{\eps_i}\to \infty$ and hence  $G$ is a \emph{group}, rather than just a pseudo-group. Further,  the maps $\pi_i$ which are almost Riemannian submersions remain almost Riemannian submersions after rescaling and hence subconverge to a submetry $\pi\co Y\to \R^k$ which by the Splitting Theorem implies that $Y$ is isometric to $F\times \R^k$. Further, by construction $F$ is compact and has diameter 1.

 Since the projection map $Z\to Y$ is a submetry and lines in $Y$  lift to lines in $Z$, the  Splitting Theorem implies that $Z\cong Z_0\times \R^k$ where $G$ acts on $Z$ by just acting on $Z_0$ and acting trivially on the $\R^k$ factor. This means that $Z_0/G\cong F$. Since $Z=\R^n$ we must have $Z_0=\R^{n-k}$.

Let $G_0$ be the connected component of identity of $G$.
Then by \cite[Corollary 4.2]{FuYa} the group $G$ has virtually abelian components group  $\Gamma=G/G_0$ and contains a finite index subgroup $\hat G$ in $G$ (containing $G_0$) such  that $\hat G/ G_0$ is free abelian of finite rank and $Z_0/\hat G$ is a flat torus of positive dimension $m$.  Then the projection map $\R^{n-k}\to T^m$ is Riemannian submersion and hence a fiber bundle. The long exact homotopy sequence of this bundle shows that $\pi_0$ of the fibers must be infinite and hence
$\Gamma$ is \emph{infinite}.

The topology of the Lie group $\Isom(\R^{n-k})$ (and of its closed subgroup $G$) is the same as the one induced by the left invariant distance $d(h_1,h_2)=\max_{|x|\le 1} |h_1(x)-h_2(x)|$ where $h_1,h_2\in \Isom(\R^{n-k})$.  Since $G$ is a Lie group it easily follows that there is $\eta>0$ such that the distance between any two different connected components of $G$ is  $\ge \eta$. Indeed, if not there exist a sequence  of elements $h_j,h_j'$ in different components of $G$ with $d(h_j,h_j')\to 0$ as $j\to\infty$. But by left invariance this gives $d((h_i')^{-1}h_i,\mathrm{Id})\to 0$ which implies that $(h_j')^{-1}h_j\in G_0$ for large $j$. This is a contradiction and therefore
there is $\eta>0$ such that for any $h_1,h_2$ in different components of $G$ we have that

\begin{equation}\label{orbit-separate}
\sup_{|x|\le 1} |h_1(x)-h_2(x)|>\eta
\end{equation}

Fix an $R>1$.
Equation  \eqref{orbit-separate} implies that for all large $i$ and  for any element of $g_i\in \Gamma_i(R)$ there is a \emph{unique} coset  $g\cdot  G_0$ such that $g_i$ is equivariantly GH-close to some element of that coset on $B_R(0)\subset \R^n$.

This gives a well defined map  $\rho_{i,R}\co \Gamma_i(R)\to \Gamma$  which by uniqueness of $\rho_{i,R}$ is a pseudo-group homomorphism. Also, since $\Gamma_i$ converges to $G$ equivariantly, the image of $\rho_{i,R}$ contains $\Gamma(R/2)$  with respect to the action of $\Gamma$ on $\R^n/G_0$.

Choose $R\gg 1$ large enough so that  $\Gamma(R/2)$ contains a generating set of $\Gamma$. Then since $\Gamma$ is a group and not just a pseudo-group we get that there is a group epimorphism $\bar \Gamma_i(R)\onto \Gamma$ (cf.  \cite[Theorem 3.10]{FuYa}).


Since  $ \bar \Gamma_i(R)\cong \pi_1(F_i)\cong \pi_1(M_i)$ we get an epimorphism $\rho_i\co \pi_1(M_i)\onto \Gamma$.

Let us now go back to the original unrescaled sequence $(M_i,p_i)\to (\R^k,0)$. By above $\Gamma$ contains a normal subgroup of finite index which is free abelian of positive rank. We can pass to the cover $\check M_i$ of $M_i$ corresponding to this subgroup. Since the order of this cover does not depend on $i$ it's easy to see that $\check M_i$ still satisfy the assumptions of the Theorem. Thus without loss of generality we can assume that $\check M_i=M_i$ and $\Gamma\cong \Z^l$  for some $l>0$ to begin with.
Choose any sequence  of natural numbers $m_i$ and 
consider the covers $\bar M_i$ of $M_i$ corresponding to the normal subgroup  $\rho_i^{-1}(\Z^{l-1}\times 2^{m_i}\Z)\subset \pi_1(M_i,p_i)$. 
We have that the covers  $u_i\co\bar M_i\to M_i$ have abelian  covering groups $H_i\cong \Z_{m_i}$. By passing to a subsequence we can assume  that $(\bar M_i, H_i,\bar p_i)\to (X,H, p)$.
Let $\bar \pi_i=\pi_i\circ u_i$ and let $\bar F_i=\bar \pi_i^{-1}(0)=u^{-1}(F_i)$. Then we still have that $\bar F_i\hookrightarrow \bar M_i$ is a homotopy equivalence. By choosing $m_i$ appropriately we can make sure that $1\le \diam \bar F_i\le 2$. 
By slightly rescaling the sequence we can assume that $\diam \bar F_i=1$ for all $i$.

Since $u_i$ is a Riemannian submersion we have that $\bar \pi_i\co M_i\to \R^k$ have the same almost Riemannian submersion properties as $\pi_i$. By above $\bar \pi_i$ converge to a submetry $\bar\pi\co X\to \R^k$ with compact fibers over $0$ of diameter 1. Therefore, by the Splitting Theorem
$X\cong K\times \R^k$  and $\diam K=1$. Moreover, as before $H$ acts on $K\times \R^k$ diagonally and trivially on the Euclidean factor. As the limit of abelian groups the group $H$ is   a compact abelian Lie group. Further, by construction $(K\times \R^k)/H=\R^k$ i.e. the action of $H$ on $K$ is transitive. This implies that $K$ is a flat torus $T^q$ for some $q>0$. 

To summarize, we have that $(\bar M_i, \bar p_i) \to (\R^k\times T^q, 0)$ for any lift $\bar p_i$ of $p_i$ and the almost Riemannian submersions $\bar \pi_i$ converge to the coordinate projection $\bar\pi\co \R^k\times T^q\to\R^k$.

Since  $\R^k\times T^q$ is a smooth Riemannian manifold the Fibration Theorem is applicable.  Hence for some $\bar r_i$ slowly converging to $\infty$ we have that $\bar\Phi_i\co B_{\bar r_i}(\bar p_i)\to \R^k\times T^q$ are $\delta_i$-almost Riemannian submersion with connected fibers provided by the Fibration Theorem which is also an $\delta_i$-GH approximation to  $B_{\bar r_i}(0,1)\subset \R^k\times T_q$. Here $\delta_i\to 0$.

 Further let $\bar \Pi_i=\bar\pi\circ \bar \Phi_i$ and let   $\bar F_i'=\bar \Pi_i^{-1}(0)$. By construction we have that $\bar \Phi_i\co \bar F_i'\onto T^q$ is a $\delta_i$-almost Riemannian submersion with connected fibers and in particular a fiber bundle since both the domain and the target are closed manifolds.

We have that both  $\bar \Pi_i$ and $\bar \pi_i$ are $\delta_i$-almost Riemannian submersions  and both  $\bar \Pi_i$ and $\bar \pi_i$ converge to $\bar \pi$. Note that the diameters of the fibers of $\bar \Phi_i$ go to $0$. Hence by adjusting $\delta_i$ if necessary we can assume that diameters of the fibers of  $\bar \Phi_i$  are $\le \delta_i$.

Therefore the fibers  $\bar F_i'=\bar \Pi_i^{-1}(0)$ and  $\bar F_i=\bar \pi_i^{-1}(0)$ are Hausdorff close in $\bar M_i$ and by the same argument as in Lemma~\ref{lem-hom-fiber} there is a homotopy equivalence
$\bar F_i'\to \bar F_i$ which homotopy commutes with the inclusions of $\bar F_i$ and $\bar F_i'$ into $\bar M_i$. Since both  $\bar F_i$ and $\bar F_i'$ are connected, from the Fibration Theorem we have an epimorphism $\bar \rho_i\co\pi_1(\bar F_i)\cong \pi_1(F_i')\to \pi_1(T^q)=\Z^q$.  Let $M_i'$ be the cover of $\bar M_i$ corresponding to $\ker\bar\rho_i$. The maps $\bar\Phi_i\co \bar M_i \to \R^k\times T^q$ lift to the maps $\Phi_i'\co M_i'\to \R^k\times \R^q$. The lifted maps remain $\delta_i$-almost Riemannian submersions with $\delta_i$-small fibers on $B_{\bar r_i}(p_i')$ where $p_i'$ is any lift of $\bar p_i$. Since $\bar r_i\to\infty$ and $q>0$ we have constructed a cover $M_i'$ of $M_i$ which pointed converges to a Euclidean space of dimension $>k$. Hence the induction assumption applies to $M_i'$ .

This finishes the proof of the Induction Theorem.

\end{proof}
\begin{remark}\label{rem-non-complete-induction}

Since in the proof of the Induction Theorem we deal with pointed convergence we don't actually need the manifolds $M_i$ to be complete. As in the case of Local and Global Fibration Theorems (see remarks ~\ref{non-complete-loc-fibra} and ~\ref{non-complete-global-fibra}) it's enough to know that the closed balls $\bar B_{R_i}(p_i)$ are compact fo some $R_i\to \infty$. Alternatively, one can assume that $M_i$ are complete but the curvature bounds only hold on $\bar B_{R_i}(p_i)$.
\end{remark}

Observe that an $\eps$-almost Riemannian submersion $\pi\co M_1^n\onto M_2^n$  between two manifolds of the same dimension is locally $(1+ \eps)$-Bilipschitz.
Further, if the fibers are small relative the injectivity radius of the target then by Lemma~\ref{lem-hom-fiber} the fibers are connected and hence are points. That means that $\pi$ is  $(1+ \eps)$-Bilipschitz globally and not just locally.

 Therefore
by an argument by contradiction the Induction Theorem applied to $k=0$  together with the Fibration Theorem  immediately yield 

\begin{theorem}[Asphericity Theorem]\label{asphericity-thm}
Given $N>1$ there exist $\vk(\eps|N)\to 0$ and $R(\eps|N)\to \infty$ as $\eps\to 0$ such that if $(M^n,g, e^{-f}\cdot \mathcal H_n)$ with $n\le N$ satisfies
$\sec\le\eps, \diam\le \eps $ and $\Ric_{f,N}\ge -\eps$ then $M$ is aspherical and  and  every ball of radius $R(\eps)$ in the universal cover $\tilde M$ is  $\vk(\eps|N)$-Lipschitz-Hausdorff close to the ball in $\R^n$ of the same radius.
\end{theorem}
\begin{remark}
In the setting of the Asphericity Theorem
Klingenberg's Lemma applied to  the universal cover $\tilde M$ implies that the injectivity radius of $\tilde M$ is $\ge \Vk(\eps|N)$.
\end{remark}
\section{Proof of the main theorem up to homeomorphism}\label{sec:main-homeo}
The Asphericity Theorem implies the Main Theorem up to homeomorphism using a combination of several known results. 
First we need the following Lemma 
\begin{lem}\label{lem-dens-semiconcave}
Let $\eps>0$ and let $(M^n,g,e^{-f} \cdot \mathcal H_n)$ be a complete Riemannian manifold satisfying $f\ge 0, \sec_M\le \eps$ and $\Ric_{f,N}M\ge -\eps$ for some $N\ge n$.

Then $e^{-f/N}$ is $\vk(\eps|N)$-concave.
\end{lem}
\begin{proof}
By  ~[Proposition 6.8]\cite{Kap-Kell-Ket-19} there is a $C=C(N)$ such that if $(X,d,v\cdot \mathcal H_n)$ with $v\le 1$ is $\CAT(1)$ and $\RCD(-1,N)$ then $v^{1/N}$ is $C(N)$-concave.
The claim of the lemma follows by rescaling.
\end{proof}
\begin{proof}[Proof of the Main Theorem up to homeomorphism]

By the Asphericity Theorem there exists $\eps=\eps(N)$ such that if $(M^n,g,e^{-f} \cdot \mathcal H_n)$ is a weighted closed Riemannian manifold with $n\le N$ and $\sec\le\eps, \diam\le \eps $ and $\Ric_{f,N}\ge -\eps$ then $M$ is aspherical. 
By increasing $N$ by at most 1 we can assume that $N$ is an integer. Let $q=N-n$. We can assume $q\ge 1$ since when $q=0$  as was the explained in the introduction the main Theorem reduces to Gromov's Almost Flat  Manifolds Theorem.

Given such $M$  let $u=e^{-f/q}$. We can rescale the weight $u$ to have $\max u=1$. Then by Lemma~\ref{lem-dens-semiconcave}  $u^{q/N}$ is $\vk(\eps|N)$-concave. This together with the fact that $\diam_M\le \eps$ easily implies that $u$ is almost constant in $C^1$ i.e.
\begin{equation}\label{eq-almost-const}
 1-\vk(\eps|N)\le u\le 1\quad \text{ and }\quad |\nabla u|\le \vk(\eps|N). 
\end{equation}

Further, since $q<N$ by $\vk(\eps)$-concavity of $u^{q/N}$, the chain rule and  \eqref{eq-almost-const} we  get that $u$ is also  $\vk(\eps|N)$-concave. In particular
\begin{equation}\label{eq:simi-super-harm}
\Delta u\le \vk(\eps|N)
\end{equation}
 Let $T^q_\eps=\R^q/\frac{\eps}{\sqrt q}\Z^q$ be the square flat torus of diameter $\eps$. Consider the warped product $(E, h)=M\times_u T_\eps$ with $h=g+u^2 ds^2$ where $ds^2$ is the canonical flat metric on $T^q_\eps$.

Then by the O'Neil formulas ~\cite{on} the Ricci tensor of $M\times_u T_\eps$  has the following form. Let  $ X,Y$ be  horizontal and $V,W$ are vertical tangent vectors at a point in $E$.
\begin{align}
\Ric(X,Y)=\Ric^M(X,Y)-\frac{q}{u}\Hess_u(X,Y)=\Ric^M_{f,N}(X,Y)\\
\Ric (X,V)=0\\
\Ric(V,W)=\Ric^{T_\eps}(V,W)-\langle V,W\rangle_{\bar g}\cdot (\frac{\Delta u}{u}+\frac{q-1}{u^2}||\nabla u|^2_g)
\end{align}
Since $T_\eps$ is flat and hence $\Ric^{T_\eps}(V,W)=0$ the assumptions of the theorem together with inequalities  \eqref{eq-almost-const} and \eqref{eq:simi-super-harm} imply (cf.  \cite[Theorem 3]{lobaem}) that  $E$ satisfies $$\Ric\ge -\vk(\eps|N).$$
It also obviously has $\diam \le 2\eps$.

By the Margulis Lemma ~\cite[Theorem 1]{kapovitchwilking} if $\eps$ is chosen small enough then $\pi_1(M\times T)$ is virtually nilpotent.
Hence the same is true for $M$. Since $M$ is aspherical, by a combination of known results this implies that $M$ is \emph{homeomorphic} to an infranilmanifold.
Indeed. Since $M$ is aspherical $\pi_1(M)$ is torsion free virtually nilpotent of rank $n$.
By a result of  Lee and Raymond \cite{LR} any such group must be
 isomorphic to the fundamental group of a compact infranilmanifold. Now the claim follows by the Borel  conjecture which asserts that  the homeomorphism type of a closed aspherical manifold is determined by its fundamental group. The Borel conjecture is open in general but it is known to hold for virtually nilpotent fundamental groups. In dimensions  above  4 this follows by work of Farrell and Hsiang ~\cite{FH}.
 The $4$-dimensional case follows from work of Freedman--Quinn \cite{FQ}. 
Lastly, the 3-dimensional case follows from Perelman's solution of the geometrization conjecture. The 2-dimensional case easily follows from the classification of closed 2-manifolds.
\end{proof}
The above proof easily generalizes to show
\begin{theorem}\label{infranil-fibers}
Let $( M_i^n, p_i)\to (B^m, p)$ in  pointed Gromov\textendash Hausdorff topology where $B^m$ is smooth and compact and all $(M_i, g_i, e^{-f_i} \mathcal H_n)$ satisfy \eqref{eq:cd+cba}.
Let $\pi_i\co M_i\to B$ be the almost Riemannian submersions provided by the fibration theorem. Then for all large $i$ the fibers $F_i=\pi_i^{-1}(p)$ are homeomorphic to compact infranilmanifolds.
\end{theorem}
\begin{proof}
Fix a small $0<\delta\ll \injrad(B)$. Then $(B_\delta(p_i), p_i)\to B_\delta(p),p)$. By rescaling by $\lambda_i\to\infty$ very slowly we get that $(\lambda_iB_{\delta}(p_i),p_i)\to (\R^m,0)$ and $\lambda_i \diam F_i$ still goes to $0$. By  Remark~\ref{rem-non-complete-induction} the Induction Theorem still applies to this convergence even though $\lambda_iB_{\delta}(p_i)$ is not complete.  This gives that $F_i$ is aspherical for all large $i$.
It's easy to deform the Riemannian metrics on $\lambda_iB_{\delta}(p_i)$ to be complete and agree with the original metrics on $\lambda_iB_{\delta/2}(p_i)$. Let $N_i$ be the resulting complete Riemannian manifolds.
 Note that we do not require \emph{any} curvature bounds on the new metrics outside $\lambda_iB_{\delta/2}(p_i)$. Let $E_i=N_i\times_{u_i}T^q_{\eps_i}$ where $q=N-n$ (we can again assume that $N$ is an integer) $u_i=e^{-f_i/q}$ and $\eps_i\to 0$. Then as before we get that $\Ric_{E_i}\ge -\nu_i\to 0$ on $B_{R_i}(q_i)$ where $q_i=(p_i,1)\in N_i\times T^q$ and  $R_i=\lambda_i\cdot \delta/100\to\infty$. Even though the lower Ricci bound on $E_i$ doesn't hold globally 
 the Margulis Lemma  \cite[Theorem 1]{kapovitchwilking}  still applies in this situation at $q_i$. It says that for some universal $\eps=\eps(N)$ the image of  $\pi_1(B_\eps^{E_i}(q_i))$ in $\pi_1(E_i)$ is virtually nilpotent. By the Fibration Theorem this image is isomorphic to $\pi_1(F_i\times T^q)\cong \pi_1(F_i)\times \Z^q$ and therefore  $\pi_1(F_i)$ is virtually nilpotent for all large $i$.
Since $F_i$ is aspherical this again implies that $F_i$ is homeomorphic to a closed infranilmanifold by the same argument as before.
\end{proof}
\begin{remark}
It should certainly be true that in the above theorem the fibers are diffeomorphic to infranilmanifolds and not  just homeomorphic. However, at the moment we only have a proof of this fact in the special case when $B$ is a point, i.e. in the setting of the Main Theorem. The proof  uses a Ricci flow argument from the next section. 
We do note that in general it is known \cite{FKS-exotic-nilmanifolds} that for $n\ne 4$ a finite cover of a smooth $n$-manifold which is homeomorphic to an infranilmanifold is diffeomorphic to a nilmanifold. Thus for $n-m\ne 4$ the diffeomorphism statement in the above theorem holds for finite covers of the the fibers $F_i$.
\end{remark}


    \section{Smoothing via Ricci flow}\label{sec:ricci flow}
    With Asphericity Theorem~\ref{asphericity-thm}  in hand we can show that Ricci flow turns  the warped products $E=M\times _{e^{-f/q}}T^q_\eps$ constructed in the previous section into almost flat manifolds in classical sense in uniform time which easily implies the Main Theorem. The Ricci flow claim  follows  from the proof (but not the statement) of the following theorem due to Huang, Kong, Rong and Xu
    
    \begin{theorem}\label{noncol-ric-alm-flat} \cite[Theorem A]{Huang-Kong-Rong-Xu}
        For any natural $n>1$ and real $v>0$ there exists $\eps=\eps(n,v)$ such that if $(M^n,g)$ is a closed Riemannian manifold satisfying $\Ric_M\ge -1, \diam (M)\le \eps$ and $\vol(B_1(\tilde p))\ge v$ for some point $\tilde p$ in the universal cover $\tilde M$ of $M$ then $M$ is diffeomorphic to a compact infranilmanifold.
    \end{theorem}

 For reader's convenience we present a direct short proof not relying on Theorem~\ref{noncol-ric-alm-flat}. It uses a somewhat different argument for the crucial upper diameter bound than that in \cite[Lemma 1.11]{Huang-Kong-Rong-Xu} which in turn relied on  ~\cite[Lemma 2.10]{Chen-Rong-Xu}.

Our proof also \emph{does not} use the main result of the previous section (that the Almost Flat Manifolds Theorem holds up to homeomorphism). 

    \begin{theorem}\label{thm:almost-flat-Ricci-flow}
For any integer $N>1$ there exist $\eps_0=\eps_0(N), c(N)>0$ such that the following holds.

For any $0<\eps<\eps_0$ if $(M^n,g,e^{-f} \cdot \mathcal H_n)$ is a weighted closed Riemannian manifold with  $\sec\le\eps, \diam\le \eps, n<N $ and $\Ric_{f,N}\ge -\eps$ then $(E=M\times _{e^{-f/q}}T^q_\eps, h)$ where $q=N-n$ admits a Ricci flow $h_t$ for $0\le t\le 1$ such that
$\diam (E, h_1)\le c(N)$ and $|\Rm( (E, h_1)|\le {\vk(\eps|N)}$. In particular, when $\eps$ is sufficiently small $(E, h_1)$ (and hence $(M,g_1)$) is almost flat in Gromov's sense. In particular $M$ is diffeomorphic to an infranilmanifold.

    \end{theorem}
    
    We will need Perelman's   Pseudolocality Theorem. Perelman proved it for closed manifolds in \cite{Per-entropy}. It was later shown to hold for complete manifolds in \cite{chau-pseudolocality} (cf. \cite{chow-ricci-III})
    \begin{theorem}[Perelman's Pseudolocality]\label{per:pseudo-locality}
    Given natural $n>1$ for any $A>0$ there exist $\delta_0>0,\eps_0>0$ such that the following holds. Suppose we have a smooth solution to the Ricci flow $(M^ng_t),     0\le t\le (\eps r)^2$ 
     a complete solution of the Ricci flow with bounded curvature, where $0<\eps<\eps_0, r>0$. Let $x_0\in M$. Suppose scalar curvature satisfies
     
     \[
     S(x,0)>-\frac{1}{r^2} \quad \text{  for any } x\in B^{g_0}_r(x_0)
     \]
     
     and  the $\delta_0$-almost Euclidean isoperimetric inequality holds in $B^{g_0}_r(x_0)$.
     
     Then we have
     
     \[
     |\Rm(x,t)|\le \frac A t +\frac{1}{(\eps_0 r)^2}
     \]
     
     for all $x\in B^{g_t}_{\eps_0 r}(x_0)$ and $0<t< (\eps r)^2$.
    
    Furthermore for all $x\in B^{g_t}_{\eps_0 r}(x_0)$ and $0<t< (\eps r)^2$ and any $\rho<\sqrt t$ it holds that
    
    \[
    \vol_t B^t_\rho(x)\ge c(n)\rho^n\quad \text{ where $c(n)$ is some universal constant.}
    \]
    \end{theorem}
    We will also need the following distance distortion estimate due to Hamilton
    
    \begin{theorem}\label{ham-distort}\cite[Theorem 17.2]{Hamilton-II}\cite[Corollary 27.16]{Lott-Kleiner}
    
    Let $(M^n,g(t))$ be a Ricci flow such that  $g(t_0)$ is complete and satisfies $\Ric_{t_0}\le K(n-1)$ for some $K\ge 0$. Then for any $x,y\in M$ we have
    
    \[
    \frac{ d}{dt}|_{t=t_0-}\,d_t(x,y)\ge -c(n)\sqrt K
    \]
    
    for some universal constant $c(n)$.
    \end{theorem}

    \begin{proof}[Proof of Theorem~\ref{thm:almost-flat-Ricci-flow}]

    Let $u=e^{-f/q}$.  As in Section~\ref{sec:main-homeo} we consider the warped product $(E,h)=M\times_u T_\eps$. Consider the Ricci flow $h_t$ on $E$ given by $\frac{\partial h_t}{\partial t}=-2\Ric_{h_t}$.

    As was shown by Lott $h_t$~\cite{lott-dim-red,Lott-optimal} remains a warped product metric $(E,h)=(M, g_t)\times_{e^{-f_t/q}} T_\eps$ where $g_t, f_t$ evolve by  a coupled system
    
    \begin{equation}
    \begin{cases}
    \frac{\partial g_t}{\partial t}=-2\Ric_{g_t,f_t, N}\\
     \frac{\partial f_t}{\partial t}=\Delta_tf_t-|\nabla f_t|_t^2
    \end{cases}
    \end{equation}
    
    We do not make direct use of this system other than using the fact that $h_t$ remains a warped product metric for all $t$.

    Next, as in Section~\ref{sec:main-homeo}, rescale $u$ to have $\max u=1$.  Then  $u$ is almost constant in $C^1$ by \eqref{eq-almost-const}.  Combined with the Asphericity Theorem~\ref{asphericity-thm} this implies that 
 there is $R(\eps)\to \infty$ such that  for \emph{any} $\tilde p\in \tilde  E$ the ball $B_R(\tilde p)$ is $\vk(\eps)$-Lipschitz-Hausdorff close to the corresponding ball in $\R^N$. In particular its isoperimetric constant is $\vk(\eps)$-close to the Euclidean one. 
 Fix $0<A\le 1$.   Let $\eps_0(A),\delta_0(A)$ be provided by the Pseudolocality Theorem.
 We have that for all small $\eps$ any ball $B_{R(\eps)}(x)$ in $\tilde E$ satisfies the $\delta_0(A)$-Euclidean isoperimetric inequality and $Scal_{\tilde E}\ge -\vk(\eps)$. By possibly making $R(\eps)$ smaller but still very slowly going to infinity as $\eps\to 0$ we can assume that $\frac{1}{R(\eps)^2}>\vk(\eps)$.

  Therefore Pseudolocality Theorem applies with $r=R(\eps)$ if $\eps$ is small enough.
  
   Let $T_{\eps,A}= (\eps_0(A) R(\eps))^2$.  Note that for any fixed $A$ we have that $T_{\eps,A}\to\infty$ as $\eps\to 0$.

   Let $T$ be the maximal time interval of existence of $h_t$. Suppose $T< T_{\eps,A}$.
  
By pseudo-locality  for any $x\in \tilde E$  and $0<t<T$ we have that

      \begin{equation}\label{eq:pseudo-loc}
     |\Rm(x,t)|\le \frac A t +\frac{1}{T_{\eps,A}}
     \end{equation}
Since this holds for arbitrary $x\in E$ we have that $\limsup_{t\to T_{-}}\sup_{\tilde E}|\Rm(x,t)|<\infty$. This implies that the Ricci flow $h_t$ can be extended past $T$. Therefore $T\ge  T_{\eps,A}$  for all small $\eps$. 

Again by pseudolocality we have that for any $0<t<T_{\eps,A}$ and any $\rho\le\sqrt t$ it holds that

\begin{equation}\label{kappa-noncol}
\vol_t(B^t_\rho(x))\ge c_1(N)\rho^N
\end{equation}
 
 For any fixed $A$ we have that for all small $\eps$ it holds that $\frac{1}{T_{\eps,A}}\le A$.

By \eqref{eq:pseudo-loc} we have that

       \begin{equation}\label{eq-almost-flat}
     |\Rm(x,1)|\le 2A\quad \text{ for any } x\in \tilde E \quad \text{ and all small } \eps.
     \end{equation}
 
 and 
 
        \begin{equation}
     |\Rm(x,t)|\le \frac{1}{2t}+\frac 1 2\le \frac{1}{t} \quad \text{ for any } x\in \tilde E \quad \text{  any } 0<t\le 1 \quad  \text{ and all small } \eps.
     \end{equation}
 
 Applying Hamilton's distance distortion estimate (Theorem~\ref{ham-distort}) this gives that for any $x\in \tilde E$ and any $0<t\le 1$ and all small $\eps$ it holds
 
  \[
    \frac{ d}{dt}|_{t=t_0-}\,d_t(x,y)\ge -\frac{c_2(N)}{\sqrt t}
    \]

 Integrating this in $t$  from $0$ to $1$ gives
 
 \begin{equation}\label{dist-est-2}
d_1(x,y)\ge d_0(x,y)-c_2(N)
 \end{equation}
 
 for any $x,y\in \tilde E$ provided $\eps$ is small enough.  Without loss of generality $c_2(N)\ge 1$.

 Let us denote $10c_2(N)$ by $l$. Fix $A=1$. Then for all small $\eps$ we have that $T_1,\eps> 100$. Therefore by \eqref{kappa-noncol}  for any $x\in \tilde E$ and $t=1$ it holds that
 \begin{equation}\label{vol-ball-below}
 \vol_1(B^1_1(x))\ge c_1(N).
\end{equation}

Let $U=B_{l}(p)$ for some $p\in \tilde E$.  Recall that $\diam E<2\eps$. Since $r\gg 2\eps>\diam E$ we have that $\pi(U)=E$ where $\pi\co \tilde E\to E$ is the universal covering map. Let $V=B_{10l}(p)$. Since $B_{R(\eps)}(p)$ is $\vk(\eps)$-Lipschitz-Hausdorff close to $B_{R(\eps)}(0)\subset \R^N$, for any $x\in V$ there is a topological $(N-1)$ sphere $\hat S_l(x)$ which is  $\vk(\eps)$-close to the metric sphere $S_l(x)$ and which separates $\tilde E$ into two connected components.

For any $y\in \hat S_l(x)$ by the distortion estimate \eqref{dist-est-2} it holds that $d_1(x,y)\ge l-\vk(\eps)-c_2(n)> 0.8l$. Since  $S_l(x)$ separates $\tilde E$ this implies that $B^1_{0.8l}(x)$ is completely contained in the bounded component of 
$\tilde E\backslash S_l(x)$ which is contained in $B^0_{1.1l}$.
Thus, for any $x\in V$ we have

\begin{equation}\label{ball-incl}
B^1_{0.8l}(x)\subset B^0_{1.1l}(x)
\end{equation}

Recall that under the Ricci flow $d\vol_t'=-\Scal_td\vol_t$. Since the minimum of scalar curvature on $E$ does not decrease along the Ricci flow we have that $\Scal(x,t)\ge -1$ for all $x\in\tilde E, t\ge 0$. Hence
$\frac{d}{dt}\vol_t(V)\le \vol_t(V)$ for all $t$ and therefore $\vol_t(V)\le e^t\vol_0(V)$ and in particular $\vol_1(V)\le 3\vol_0(V)$. Since $\Ric^0\ge -1$ by absolute volume comparison at time $0$ we get that $\vol_1(V)\le c_3(N)$ for some universal constant $c_3(N)$. Suppose $\diam_1(U)\ge m\cdot 10 l$ for some integer $m$. Here and in what follows the diameter is measured with respect to the ambient metric.

Since $U$ is connected this implies that $U$ contains $m$ points $x_1,\ldots x_m$ such that the balls $B^1_{0.8l}(x_i), i=1,\ldots m$ are disjoint.

By \eqref{ball-incl} we have that each $B^1_{0.8l}(x_i)$ is contained in  $B^0_{1.1l}(x_i)$ which is contained in $V$.  Therefore $\vol_1(\cup_iB^1_{0.8l}(x_i))\le \vol_1(V)\le c_3(N)$.

On the other hand $0.8l\ge 1$ and therefore $\vol_1(B^1_{0.8l}(x_i))\ge c_1(N)$ by \eqref{vol-ball-below}. Since all these balls are disjoint, combining the last two estimates we get

\[
m\cdot c_1(N)\le \vol_1(\cup_iB^1_{0.8l}(x_i))\le \vol_1(V)\le c_3(N), \quad m\le \frac{c_3(N)}{c_1(N)}
\]

and hence

\[
\diam_1(U)\le ( \frac{c_3(N)}{c_1(N)}+2)10l
\]
Since $\pi(U)=E$ this implies that

\[
\diam_1(E)\le ( \frac{c_3(N)}{c_1(N)}+2)10l
\]
as well. Note that this is a universal constant which does not depend on $\eps$. Combining this with \eqref{eq-almost-flat} we conclude that $E$ (and hence $M$) is almost flat in Gromov's sense for all small $\eps$.
By Gromov's Almost Flat Manifolds Theorem~\cite{Gr-almost-flat,BuKa, Ruh-almost-flat} this implies that $M$ is infranil provided $\eps$ is small enough.

    \end{proof}
    

    In   \cite{Kap-Kell-Ket-19} a structure theory of $\RCD+\CAT$ spaces was developed. In particular it is proved that that if $(X,d,m)$ is $\RCD(\lk.N)$ and $\CAT(\uk)$ with $N<\infty$ then it's a topological manifold with boundary  of dimension $\le N$ and the manifold interior is a smooth $C^1$-manifold  with a $C^0\cap BV$ Riemannian metric. Further a geometric boundary $\bCAT X$ of such space  $X$ was defined and shown to be equal to the manifold boundary $\partial X$. As these results are local they apply in equal measure to spaces  which are $\RCD(\lk,N)+\CBA(\uk)$ with $N<\infty$. By above  when the boundary is empty any such space is a smooth manifold. It is therefore natural to wonder if Theorem \ref{thm:almostflat} holds in this generality. We conjecture that it does.
    \begin{conjecture}
For any $1<N<\infty$ there exists $\eps=\eps(N)$ such that the following holds.

If $(X,d,m)$ is an $\RCD(\eps, N)$ space such that $(X,d)$ is $\CBA(\eps)$  and $\bCAT X=\varnothing$ and $\diam X\le \eps$ then $X$ is diffeomorphic to an infranilmanifold of dimension $\le N$.
    \end{conjecture}
    The proofs of the  Induction Theorem and of the Asphericity Theorem  go through with minimal changes in this setting. However, generalizing the Ricci flow part of the proof would require more work due to the low regularity of the initial Riemannian metric. We fully expect this to be possible however.

    \section{Almost  flat mixed curvature metrics on $T^2$ without uniform lower sectional curvature bounds }\label{sec:example}
    
    The aim of this section is to prove the following
    \begin{lemma}\label{lemma-ex}
    There exists a sequence of weighted metrics  $(T^2, g_i,f_i)$ such that for any $N>2$  it holds $\sec g_i\le 1/i, \Ric_{f_i,N}\ge-\eps_i(N),\diam (T^2,g_i)\le 10$ where $\eps_i(N)\to 0$ as $i\to \infty$ and such that
    $\min \sec_{g_i}\to\ -\infty$ as $i\to\infty$.
    \end{lemma}    
  This shows that the assumptions of Theorem~\ref{thm:almostflat} do not imply any uniform lower sectional curvature bounds on $M$. The example we construct is for illustrative purposes and the material of this section is not needed for the rest of the paper.
    \begin{proof}
    Let $0<\eps\ll 1$.  Let $g_0$ be the canonical metric on $\R^2$.
    Let $\hat h_\eps\co [0,1]\to  [0,\infty)$ be a piecewise smooth function such that $\hat h_\eps(0)=0, \hat h'(x)=\frac{x}{\eps}$ for $0\le x\le \eps^2$, $\hat h'(x)=\eps$ for $\eps^2\le x\le 0.99$ and $\hat h_\eps'(x)=0$ for $0.99\le x\le 1$.
    It's easy to smooth it out near $\eps^2$ and $0.9$ to get  a $C^\infty$ function $\check h_\eps\co  [0,1]\to  [0,\infty)$ such that
    
    \begin{enumerate}
    \item $\check h_\eps=\hat h_\eps$ on $[0,\eps^2/2]$ and $[4\eps^2,0.9]$;
    \item $|\check h_\eps(x)|\le \eps$ and $|\check h_\eps'(x)|\le \eps$ for all $x$;
    \item $\check h_\eps''\ge 0 $ on $[0,1/2]$
    \item $|\check h_\eps'|\le \eps, |h_\eps''|\le \eps$ for $x\ge 4\eps^2$
    \item $\check h_\eps''\le \frac{2}{\eps}$ for $0\le x\le 3\eps^2$
    \item $\check h''(x)\le \frac{2}{\eps}$ for all $x$.
    \end{enumerate}
    
    Next, consider $\eps x^2+\check h_\eps$. It's easy to extend it to $1,\infty$ and smooth it out near 1 such that the resulting function $h_\eps\co[0,\infty)\to[0,\infty)$ is $C^\infty$ and
    satisfies
    
    \begin{enumerate}
    \item $h_\eps=\eps x^2+\check h_\eps$ on $[0,3/4]$
    \item $|h_\eps|\le\eps, |h_\eps'|\le 10\eps,  |h_\eps''|\le 10\eps$ for $x\ge 10\eps^2$;
    \item $h_\eps(x)=0$ for $x\ge 2$.
    \end{enumerate}
    
    Note that by construction $||h||_{C^2}\le \eps$ on $[10\eps^2,\infty)$. Also $h_\eps$ is convex on $[0,3/4]$.
    
    Let $f_\eps\co\R^2\to\R$ be given by $f_\eps(x,y)=h_\eps(\sqrt{x^2+y^2})$. Then by construction $f_\eps$ is $C^\infty$ on $\R^2$ and has the following properties
    
    \begin{enumerate}
    \item\label{eq-small-ball} On $B_{\eps^2/3}(0)$ it holds that $f_\eps(x,y)=(\frac{1}{2\eps}+\eps)(x^2+y^2)$
    \item \label{hess-upper-bnd} $|\Hess_{f_\eps}|\le \frac{4}{\eps}$ on $\R^2$.
    \item \label{lower-hess} $\Hess_{f_\eps}\ge -20{\eps}\cdot g_0$ on $\R^2$
 \item     $f_\eps$ has support in $\bar B_2(0)$
 \item \label{C-2-small} $||f_\eps||_{C^2}\le 20\eps$ outside $B_{4\eps^2}(0)$.
 \item \label{str-convex}    $f_\eps$ is strictly convex on $B_{3/4}(0)$.
 
    \end{enumerate}
    
    Let $\phi_\eps=\frac{f_{\eps/10}}{1000}$. 
    
   Let $g_\eps=e^{2\phi_\eps}g_0$. Look at the weighted Riemannian manifold $M^2_\eps=(\R^2,g_\eps, e^{-f_\eps}d\vol_{g_\eps})$.
   
   Let us compute its sectional curvature and Barkty-Emery Ricci curvature. We use the same computations as in~\cite[Example 6.8]{Kap-Ket-18}.
   
    Recall that given a Riemannian manifold  $(M^n,g)$ if we change the Riemannian metric conformally $\tilde g=e^{2\phi} g$ then for any smooth function $f$ on $M$ its hessian changes by the formula

\begin{equation}\label{hess-conf}
\tilde \Hess_f(V,V)=\Hess_f(V,V)-2\langle \nabla \phi, V\rangle \langle \nabla f, V\rangle+|V|^2\cdot \langle \nabla \phi, \nabla f\rangle
\end{equation}
here and in what follows $\langle\cdot,\cdot\rangle$ and $\langle\cdot,\cdot\rangle_\sim$ are the inner products with respect to $g$ and $\tilde g$ respectively.

Also, recall that when $n=2$ Ricci tensors of $g$ and $\tilde g$ are related as follows

\begin{equation}\label{eq-ric-conf}
\tilde \Ric (V,V)=\Ric(V,V)-\Delta \phi |V|^2
\end{equation}
Then for any $N>2$ the weighted $N$-Ricci tensor of $(M^2, \tilde g, e^{-f}d\vol_{\tilde g})$ is equal to

\begin{equation}\label{N-ric-conf}
\begin{gathered}
\tilde \Ric_f^N(V,V)=\tilde \Ric(V,V)+\tilde\Hess_f(V,V)-\frac{\langle \tilde\nabla f, V\rangle_\sim^2}{N-2}=\\
=\Ric(V,V)-\Delta \phi |V|^2+\Hess_f(V,V)-2\langle \nabla \phi, V\rangle \langle \nabla f, V\rangle+|V|^2\cdot \langle \nabla \phi, \nabla f\rangle-\frac{\langle \nabla f, V\rangle^2}{N-2}
\end{gathered}
\end{equation}
 
 Let us apply the above with $\phi=\phi_\eps, f=f_\eps$.
 
 First, equation \eqref{eq-ric-conf} tells us that $\sec(M_\eps)=-\Delta\phi_\eps$. 
 By ~\eqref{str-convex}   we have that $\sec_{M_\eps}\le 0$ on $B_{3/4}(0)$ and by \eqref{C-2-small} we have that  $\sec_{M_\eps}\le \eps$ outside $B_{\eps^2}(0)$. Altogether this means that $\sec_{M_\eps}\le \eps$ everywhere.   
 
 Lastly let us examine the Bakry\textendash Emery Ricci tensor of $M_\eps$. Since $\phi_\eps$ and $f_\eps$ are $2\eps$-close to 1 in $C^1(\R^2)$, the last  3 terms in \eqref{N-ric-conf} are small for small $\eps$ and unit $V$.
 The first term vanishes since $g_0$ is flat. 
 
On $B_{\eps^2/3}$ we have that $\Hess_{f_\eps}(V,V)\ge\frac{1}{\eps}|V|^2$ by  \eqref{eq-small-ball}. Also on the same ball $\Delta\phi_\eps\le \frac{1}{3\eps}$ by \eqref{hess-upper-bnd}.
 This means that  $\Hess_f(V,V)-\Delta \phi |V|^2\ge \frac{2}{3\eps}$ on $B_{\eps^2/3}$. Also, outside  $B_{\eps^2/2}$ we have that  $\Hess_f(V,V)\ge -20\eps|V|^2$ by~\eqref{lower-hess}. Also, outside the same ball 
 $|\Delta\phi_\eps|\le 20\eps$.  Altogether this means that  $\Hess_f(V,V)-\Delta \phi |V|^2\ge -40\eps$ on all of $\R^2$.
 
 Plugging this into \eqref{N-ric-conf} we get that the Bakry\textendash Emery Ricci curvature of $M^2_\eps$ satisfies $\Ric^{g_\eps}_{f_\eps,N}\ge -c_1\cdot \eps- \frac{c_2\eps}{N-2}$ for some universal constants $c_1>0,c_2>0$.
 

 Since $f_\eps\equiv 1,\phi_\eps\equiv 1$ outside $B_2(0)$ we can  use them to produce a weighted metric with the same sectional and Bakry\textendash Emery bounds on $T^2$ of diameter $\le 10$. Since this is possible for any $\eps>0$ taking an appropriate $\eps_i\to 0$ yields the Lemma.
 
    \end{proof}
 
\small{
\bibliographystyle{my-amsalpha}

\newcommand{\etalchar}[1]{$^{#1}$}
\providecommand{\bysame}{\leavevmode\hbox to3em{\hrulefill}\thinspace}
\providecommand{\MR}{\relax\ifhmode\unskip\space\fi MR }
\providecommand{\MRhref}[2]{%
  \href{http://www.ams.org/mathscinet-getitem?mr=#1}{#2}
}
\providecommand{\href}[2]{#2}

\end{document}